\documentclass[reqno,12pt]{amsart}
\usepackage{amssymb,amsmath,amsthm,mathtools,mathrsfs,calc}
\usepackage{enumerate}
\usepackage[margin=1in]{geometry}
\usepackage{bm}

\makeatletter
\let\@@pmod\pmod
\DeclareRobustCommand{\pmod}{\@ifstar\@pmods\@@pmod}
\def\@pmods#1{\mkern4mu({\operator@font mod}\mkern 6mu#1)}
\makeatother

\def\Z{\mathbb{Z}}
\def\Q{\mathbb{Q}}

\def\H{\mathbb{H}}
\def\N{\mathbb{N}}
\def\C{\mathbb{C}}
\def\P{\mathbb{P}}
\def\F{\mathbb{F}}
\def\sQ{\mathcal{Q}}
\def\fQ{\mathscr{Q}}

\DeclareMathOperator{\im}{Im}
\DeclareMathOperator{\re}{Re}
\def\SL{{\rm SL}}

\def\GL{{\rm GL}}
\newcommand{\pfrac}[2]{\left(\frac{#1}{#2}\right)}
\newcommand{\ptfrac}[2]{\left(\tfrac{#1}{#2}\right)}
\newcommand{\pMatrix}[4]{\left(\begin{matrix}#1 & #2 \\ #3 & #4\end{matrix}\right)}
\renewcommand{\pmatrix}[4]{\left(\begin{smallmatrix}#1 & #2 \\ #3 & #4\end{smallmatrix}\right)}
\renewcommand{\bar}[1]{\overline{#1}}
\newcommand{\M}{\mathcal M}
\newcommand{\bP}{\bm{P}}
\renewcommand{\sl}{\big|}

\DeclareMathOperator{\sgn}{sgn}
\DeclareMathOperator{\Tr}{Tr}

\def\ep{\varepsilon}

\let\temp\phi
\let\phi\varphi
\let\varphi\temp

\newtheorem{theorem}{Theorem}
\newtheorem{lemma}[theorem]{Lemma}
\newtheorem{proposition}[theorem]{Proposition}
\theoremstyle{remark}
\newtheorem*{remark}{Remark}

\numberwithin{equation}{section}

\begin{document}

\title[Singular invariants]{Singular invariants and coefficients of \\ weak harmonic Maass forms of weight $5/2$}
\author{Nickolas Andersen}
\address{Department of Mathematics\\
University of Illinois\\
Urbana, IL 61801} 
\email{nandrsn4@illinois.edu}

\begin{abstract}	
We study the coefficients of a natural basis for the space of weak harmonic Maass forms of weight $5/2$ on the full modular group.
The non-holomorphic part of the first element of this infinite basis encodes the values of the partition function $p(n)$. 
We show that the coefficients of these harmonic Maass forms are given by traces of singular invariants.
These are values of non-holomorphic modular functions at CM points or their real quadratic analogues: cycle integrals of such functions along geodesics on the modular curve.
The real quadratic case relates to recent work of Duke, Imamo\=glu, and T\'oth on cycle integrals of the $j$-function, while the imaginary quadratic case recovers the algebraic formula of Bruinier and Ono for the partition function.
\end{abstract}

\maketitle

\allowdisplaybreaks

\section{Introduction}

A \emph{weak}\footnote{For brevity, we will drop the adjective ``weak'' throughout this paper.} \emph{harmonic Maass form} of weight $k$ is a real analytic function on the upper-half plane $\H$ which transforms like a modular form of weight $k$, is annihilated by the weight $k$ hyperbolic Laplacian
\begin{equation} \label{eq:def-laplacian}
	\Delta_k := -y^2 \left(\frac{\partial^2}{\partial x^2} + \frac{\partial^2}{\partial y^2}\right) + iky\left(\frac{\partial}{\partial x} + i\frac{\partial}{\partial y}\right), \qquad \tau = x+iy
\end{equation}
and has at most linear exponential growth at the cusps.
Such a form $h$ has a natural decomposition $h=h^++h^-$ into the holomorphic part $h^+$ (also called a \emph{mock modular form}) and the nonholomorphic part $h^-$. 
Let $\xi_k$ denote the differential operator
\begin{equation*}
	\xi_k := 2i y^k \overline{\frac{\partial}{\partial\bar \tau}}.
\end{equation*}
The function $g:=\xi_k h=\xi_k h^-$, the so-called \emph{shadow} of $h^+$, is a weakly holomorphic modular form (a modular form whose poles, if any, are supported at the cusps) of weight $2-k$.
See \cite{Bruinier:2004,Ono:2009} for background on harmonic Maass forms.

A natural problem is to determine what arithmetic data, if any, is encoded in the coefficients of mock modular forms.
Although such coefficients are not well understood in general, significant progress has been made recently in certain cases.
For instance, Ramanujan's mock theta functions (whose coefficients encode combinatorial information) are mock modular forms of weight $1/2$ whose shadows are linear combinations of weight $3/2$ theta series. 
Zwegers' discovery connecting mock theta functions to harmonic Maass forms \cite{Zwegers:RealAnalytic,Zwegers:thesis} has inspired many works (see \cite{Bringmann:2006,BO:Dyson,Zagier:2009} and the references therein for examples).
Recently Bruinier and Ono \cite{Bruinier:2010a} considered harmonic Maass forms whose shadows are cusp forms orthogonal to the weight $3/2$ theta series. 
They showed that the coefficients of these forms are related to central critical values and derivatives of weight $2$ modular $L$-functions.
Later Bruinier \cite{Bruinier:Periods} connected these coefficients to periods of algebraic differentials of the third kind on modular and elliptic curves.

In this paper, we investigate the arithmetic nature of the coefficients of mock modular forms of weight $5/2$ whose shadows are weakly holomorphic modular forms.
This relates to work of Duke, Imamo\=glu, and T\'oth \cite{DIT:CycleIntegrals} regarding a particularly interesting family of weight $1/2$ mock modular forms which are related to Zagier's proof \cite{Zagier:2002} of Borcherds' theorem \cite{Borcherds:infinite} on infinite product formulas for modular forms. 
Let $J$ denote the normalized Hauptmodul for $\SL_2(\Z)$ given by
\[
	J(\tau) := j(\tau) - 744 = \frac 1q + 196884q+ \ldots, \qquad q:= \exp(2\pi i \tau).
\]
For each nonzero discriminant $d\equiv 0,1\pmod{4}$, let $\fQ_d$ denote the set of quadratic forms of discriminant $d$ which are positive definite when $d<0$.
The modular group $\Gamma_1:=\mathrm{PSL}_2(\Z)$ acts on these forms, and the set $\Gamma_1\backslash \fQ_d$ is a finite abelian group.

If $Q$ is positive definite then $Q(\tau,1)$ has exactly one zero $\tau_Q$ in $\H$. For fixed $d$, the values
\begin{equation} \label{eq:singular-moduli}
	J(\tau_Q), \qquad Q\in \Gamma_1 \backslash \fQ_d
\end{equation}
are conjugate algebraic integers known as \emph{singular moduli}. 
Zagier \cite{Zagier:2002} showed that the weighted sums
\begin{equation}  \label{eq:zagier-traces}
	\sum_{Q\in \Gamma_1\backslash \fQ_{dD}} \frac{\chi_D(Q)}{w_Q} J(\tau_Q)
\end{equation}
appear as coefficients of a basis $\{g_D\}_{0<D\equiv 0,1(4)}$ for $\M_{3/2}^+$, the plus-space of weakly holomorphic modular forms of weight $3/2$ on $\Gamma_0(4)$.
Here $w_Q$ denotes the order of the stabilizer of $Q$ in $\Gamma_1$, and $\chi_D:\fQ_d\to\{-1,0,1\}$ is a generalized genus character (see Section I.2 of \cite{GKZ:Heegner}).

The coefficients of the forms in Zagier's basis $\{g_D\}$ also appear as coefficients of forms in the basis $\{f_D\}_{0\geq D\equiv 0,1(4)}$ given by Borcherds in Section 14 of \cite{Borcherds:infinite}. Borcherds showed that the coefficients of the $f_D$ are the exponents in the infinite product expansions of certain meromorphic modular forms.
In \cite{DIT:CycleIntegrals}, Duke, Imamo\=glu, and T\'oth extended Borcherds' basis to a basis $\{f_D\}_{D\equiv 0,1(4)}$ for $\mathbb{M}_{1/2}$, the space of mock modular forms of weight $1/2$ on $\Gamma_0(4)$ satisfying Kohnen's plus-space condition.
When $D>0$, the shadow of the mock modular form $f_D$ is proportional to the weakly holomorphic form $g_D$.
Whereas the forms $f_D$ for $D<0$ encode imaginary quadratic information in the form of traces of singular moduli, the forms $f_D$ for $D>0$ encode real quadratic information in the form of traces of cycle integrals of $J(\tau)$.
To explain this requires some notation.

For an indefinite quadratic form $Q$ with non-square discriminant, let $C_Q$ denote the geodesic in $\H$ connecting the two (necessarily irrational) roots of $Q$, modulo the stabilizer of $Q$. Then $C_Q$ defines a closed geodesic on the modular curve $\Gamma_1\backslash \H\cup\{\infty\}$, and the cycle integral
\begin{equation} \label{eq:cycle-integral}
	\int_{C_Q} J(\tau) \frac{d\tau}{Q(\tau,1)}
\end{equation}
is a well-defined invariant of the equivalence class of $Q$.
The beautiful result of Duke, Imamo\=glu, and T\'oth \cite{DIT:CycleIntegrals} interprets the coefficient $a(D,d)$ of $q^d$ in $f_D$ ($D>0$) in terms of the sums
\begin{equation} \label{eq:dit-traces}
	\frac{1}{2\pi} \sum_{Q\in \Gamma_1\backslash\fQ_{dD}} \chi_D(Q) \int_{C_Q} J(\tau) \frac{d\tau}{Q(\tau,1)},
\end{equation}
whenever $dD$ is not a square. 
However, their theorem does not include the coefficients $a(D,d)$ when $dD$ is a square since, in that case, the integral \eqref{eq:cycle-integral} diverges.
There are two approaches to studying these square-indexed coefficients.
The first method, by Bruinier, Funke and Imamo\=glu in \cite{BFI:Regularized}, involves regularizing the divergent integral \eqref{eq:cycle-integral}; their method generalizes the results of \cite{DIT:CycleIntegrals} to a large class of non-holomorphic modular functions.
The second method, by the present author in \cite{Andersen:Geodesics}, involves replacing $J(\tau)$ in \eqref{eq:cycle-integral} by $J_Q(\tau)$, a dampened version of $J(\tau)$ which depends on $Q$ (and for which the resulting integral converges).

Here we study $H_{5/2}^!(\ep)$, the space of harmonic Maass forms of weight $5/2$ which transform as
\[
	f(\gamma \tau) = \ep(\gamma) (c\tau+d)^{5/2} f(\tau) \quad \text{ for all }\gamma=\pmatrix abcd \in \SL_2(\Z).
\]
Here $\ep$ is the multiplier system defined by
\[
	\ep(\gamma \tau) := \frac{\eta(\gamma \tau)}{\sqrt{c\tau+d} \, \eta(\tau)},
\]
where $\eta$ is the Dedekind eta-function $\eta(\tau):= q^{1/24}\prod_{k\geq 1}(1-q^k)$.
We will show that the coefficients of forms in $H_{5/2}^!(\ep)$ are given by traces of singular invariants; that is, sums of CM values of certain weight $0$ weak Maass forms on $\Gamma_0(6)$ in the imaginary quadratic case, or cycle integrals of those forms over geodesics on the modular curve 
\[
	X_0(6) := \widehat{ \Gamma_0(6) \backslash \H }
\] 
in the real quadratic case.

In \cite{Ahlgren:2013aa}, the author and Ahlgren constructed an infinite basis $\{h_m\}_{m\equiv 1(24)}$ for $H_{5/2}^!(\ep)$. 
For $m<0$ the $h_m$ are holomorphic, while for $m>0$ they are non-holomorphic harmonic Maass forms whose shadows are weakly holomorphic modular forms of weight $-1/2$. 
The first element, ${\bP}:=h_1$, has Fourier expansion\footnote{Note that here, and in \eqref{eq:h-m-neg} and \eqref{eq:h-m-pos}, we have renormalized the functions ${\bP}$ and $h_m$ from \cite{Ahlgren:2013aa}. See \eqref{eq:h-m-normalize} below.}
\begin{equation} \label{eq:P-exp}
	{\bP}(\tau) = i \, q^{1/24} + \sum_{0<n\equiv 1(24)} n\, p(1,n) q^{n/24} 
	- i\beta(-y)q^{1/24} + \sum_{0>n\equiv 1(24)} |n|\, p(1,n) \beta(|n|y) q^{n/24},
\end{equation}
where $\beta(y)$ is the normalized incomplete gamma function
\[
	\beta(y) := \frac{\Gamma\left(-\tfrac 32, \tfrac{\pi y}{6}\right)}{\Gamma(-\tfrac32)} = \frac{3}{4\sqrt\pi}\int_{\pi y/6}^\infty e^{-t} t^{-5/2} dt.
\]
The function ${\bP}(\tau)$ is of particular interest since $\xi_{5/2}({\bP})$ is proportional to 
\[
	\eta^{-1}(\tau) = q^{-1/24} + \sum_{0>n\equiv 1(24)} p\ptfrac{1-n}{24} q^{|n|/24},
\]
where $p\pfrac{1-n}{24}$ is the ordinary partition function. Corollary 2 of \cite{Ahlgren:2013aa} shows that that for negative $n\equiv 1\pmod{24}$ we have $p(1,n)=\sqrt{|n|}\,p\ptfrac{1-n}{24}$.

Building on work of Hardy and Ramanujan \cite{HardyRamanujan}, Rademacher \cite{Rademacher1,Rademacher2} proved the exact formula 
\[
	p(k) = 2\pi (24k-1)^{-3/4} \sum_{c>0} \frac{A_c(k)}{c} I_{3/2}\pfrac{\pi \sqrt {24k-1}}{6c},
\]
where $I_{3/2}(x)$ is the $I$-Bessel function and $A_c(k)$ is a Kloosterman sum.
Using Rademacher's formula, Bringmann and Ono \cite{BO:partition} showed that $p(k)$ can be written as a sum of CM values of a certain non-holomorphic Maass-Poincar\'e series on $\Gamma_0(6)$. 
Later Bruinier and Ono \cite{BO:AlgebraicFormulas} refined the results of \cite{BO:partition}; by applying a theta lift to a certain weight $-2$ modular form, they obtained a new formula for $p(k)$ as a finite sum of algebraic numbers.

The formula of Bruinier and Ono involves a certain non-holomorphic, $\Gamma_0(6)$-invariant function $P(\tau)$, which the authors define in the following way.
Let
\[
	F(\tau) := \frac 12 \cdot \frac{E_2(\tau) - 2E_2(2\tau) - 3E_2(3\tau) + 6E_2(6\tau)}{(\eta(\tau)\eta(2\tau)\eta(3\tau)\eta(6\tau))^2} = q^{-1} - 10 - 29q - \ldots,
\]
where $E_2$ denotes the weight $2$ quasi-modular Eisenstein series
\[
	E_2(\tau) := 1-24\sum_{n=1}^\infty \sum_{d|n} d q^n.
\]
The function $F(\tau)$ is in $M_{-2}^!(\Gamma_0(6),1,-1)$, the space of weakly holomorphic modular forms of weight $-2$ on $\Gamma_0(6)$, having eigenvalues $1$ and $-1$ under the Atkin-Lehner involutions $W_6$ and $W_2$, respectively (see Section~\ref{sec:poincare} for definitions). 
The function $P(\tau)$ is the weak Maass form given by
\[
	P(\tau) 
	:= \frac{1}{4\pi}R_{-2}F(\tau) = -\left(q\frac{d}{dq} + \frac{1}{2\pi y}\right) F(\tau),
\]
where $R_{-2}$ is the Maass raising operator in weight $-2$ (see Section~\ref{sec:poincare} for details).

For each $n\equiv 1\pmod{24}$ and $r\in \{1,5,7,11\}$, let
\[
	\sQ^{(r)}_n := \left\{ [a,b,c]\in \fQ_n : 6\mid a \text{ and } b\equiv r\pmod*{12} \right\}.
\]
Here $[a,b,c]$ denotes the quadratic form $ax^2+bxy+cy^2$.
The group $\Gamma:=\Gamma_0(6)/\{\pm 1\}$ acts on $\sQ^{(r)}_n$, and for each $r$ we have the canonical isomorphism (see \cite[Section I]{GKZ:Heegner})
\begin{equation} \label{eq:quad-form-iso}
	\Gamma \backslash \sQ^{(r)}_n \stackrel{\sim}{\longrightarrow} \Gamma_1 \backslash \fQ_n.
\end{equation}
Bruinier and Ono \cite[Theorem 1.1]{BO:AlgebraicFormulas} proved the finite algebraic formula
\begin{equation} \label{eq:bo-alg-formula}
	p\ptfrac{1-n}{24} 
		= \frac 1n \sum_{Q\in \Gamma\backslash\sQ^{(1)}_{n}} P(\tau_Q),
\end{equation}
for each $0>n\equiv 1\pmod{24}$.
Therefore the coefficients of the non-holomorphic part of ${\bP}(\tau)$ are given in terms of singular invariants (note the similarity with the expression \eqref{eq:zagier-traces}). 

For the holomorphic part of ${\bP}(\tau)$ one might suspect by analogy with \eqref{eq:dit-traces} that the traces
\begin{equation} \label{eq:traces-no-deriv}
	\sum_{Q\in \Gamma\backslash \sQ_n^{(1)}} \int_{C_Q} P(\tau) \frac{d\tau}{Q(\tau,1)}
\end{equation}
give the coefficients $p(1,n)$. 
However, as we show in the remarks following Proposition \ref{prop:traces-no-deriv} below, the traces \eqref{eq:traces-no-deriv} are identically zero whenever $n$ is not a square. 
In order to describe the arithmetic nature of the coefficients $p(1,n)$, we must introduce the function $P(\tau,s)$, of which $P(\tau)$ is the specialization at $s=2$ (see Section \ref{sec:poincare} for the definition of $P(\tau,s)$). 
Similarly, for each $Q\in \sQ_n^{(1)}$ whose discriminant is a square, we will define $P_{Q}(\tau,s)$, a dampened version of $P(\tau,s)$ (see Section \ref{sec:quadratic}). 
These are analogues of the functions $J_Q(\tau)$ and are defined in \eqref{eq:def-p-v-q} below. 
Then, for each $n\equiv 1\pmod{24}$ we define the trace
\begin{equation} \label{eq:def-traces-1}
	\Tr(n) := 
	\begin{dcases}
		\frac 1{\sqrt{|n|}} \sum_{Q\in \Gamma\backslash \sQ^{(1)}_n} P(\tau_Q) & \text{ if } n<0, \\
		\frac 1{2\pi} \sum_{Q\in \Gamma\backslash \sQ^{(1)}_n} \int_{C_Q} \left[\frac{\partial}{\partial s}P(\tau,s) \bigg|_{s=2}\right] \frac{d\tau}{Q(\tau,1)} & \text{ if } n>0 \text{ is not a square}, \\
		\frac 1{2\pi} \sum_{Q\in \Gamma\backslash \sQ^{(1)}_n} \int_{C_Q} \left[\frac{\partial}{\partial s}P_{Q}(\tau,s) \bigg|_{s=2}\right] \frac{d\tau}{Q(\tau,1)} & \text{ if } n>0 \text{ is a square}.
	\end{dcases}
\end{equation}

We have the following theorem which provides an arithmetic interpretation for the coefficients of ${\bP}(\tau)$ by relating them to the traces \eqref{eq:def-traces-1}. It is a special case of Theorem~\ref{thm:main} below.

\begin{theorem} \label{thm:1}
	For each $n\equiv 1\pmod{24}$ we have
	\begin{equation}
		p(1,n) = \Tr(n).
	\end{equation}
\end{theorem}

\begin{remark}
	Since $p(1,n)=\sqrt{|n|} \, p\pfrac{1-n}{24}$ for $0>n\equiv 1\pmod{24}$, Theorem \ref{thm:1} recovers the algebraic formula \eqref{eq:bo-alg-formula} of Bruinier and Ono. However, our proof is quite different than the proof given in \cite{BO:AlgebraicFormulas}.
\end{remark}

Recall that ${\bP}=h_1$ is the first element of an infinite basis $\{h_m\}$ for $H_{5/2}(\ep)$.
We turn to the other harmonic Maass forms $h_m$ for general $m\equiv 1\pmod{24}$.
For negative $m$, we have the Fourier expansion (see \cite[Theorem 1]{Ahlgren:2013aa})
\begin{equation} \label{eq:h-m-neg}
	h_m(\tau) = |m|^{3/2} q^{m/24} - \sum_{0<n\equiv 1(24)} |mn|\, p(m,n) q^{n/24}.
\end{equation}
These forms are holomorphic on $\H$ and can be constructed using $\eta(\tau)$ and $j'(\tau):=-q\frac{d}{dq}j(\tau)$.
We list a few examples here.
\begin{align*}
	23^{-\frac32}h_{-23} &= \eta \, j' \\
	&= q^{-23/24} - q^{1/24} - 196\,885 \,q^{25/24} - 42\,790\,636 \,q^{49/24} - \ldots, \\
	47^{-\frac32}h_{-47} &= \eta \, j'(j-743) \\
	&= q^{-47/24} - 2 \,q^{1/24} - 21\,690\,645 \,q^{25/24} - 40\,513\,206\,272 \,q^{49/24} - \ldots, \\
	71^{-\frac32}h_{-71} &= \eta \, j'(j^2-1487j+355\,910) \\
	&= q^{-71/24} - 3 \,q^{1/24} - 886\,187\,500 \,q^{25/24} - 8\,543\,738\,297\,129 \,q^{49/24} - \ldots.
\end{align*}
For positive $m$, we have the Fourier expansion (see \cite[Theorem 1]{Ahlgren:2013aa})
\begin{multline} \label{eq:h-m-pos}
	h_m = i\, m^{3/2} q^{m/24} + \sum_{0<n\equiv 1(24)} mn \, p(m,n) q^{n/24} \\
	- i\, m^{3/2} \beta(-m y) q^{m/24} + \sum_{0>n\equiv 1(24)} |mn|\, p(m,n) \beta(|n|y) q^{n/24}.
\end{multline}
For these $m$, the shadow of $h_m$ is proportional to the weight $-1/2$ form
\[
	g_m(\tau) := m^{-3/2} q^{-m/24} + \sum_{0>n\equiv 1(24)} |mn|^{-1/2} \, p(n,m) q^{|n|/24},
\]
with $p(n,m)$ as in \eqref{eq:h-m-neg}. 
When $mn<0$, we have the relation $p(n,m)=-p(m,n)$.
The first few examples of these forms are
\begin{align*}
	g_1 
	&= \eta^{-1} \\
	&= q^{-\frac{1}{24}} + q^{\frac{23}{24}} + 2 \,q^{\frac{47}{24}} + 3 \,q^{\frac{71}{24}} + 5 \,q^{\frac{95}{24}} + 7 \,q^{\frac{119}{24}} + 11 \,q^{\frac{143}{24}}+ \ldots, \\
	5^3 g_{25} 
	&= \eta^{-1} (j - 745) \\
	&= q^{-\frac{25}{24}}+196\,885\, \,q^{23/24}+21\,690\,645 \,q^{47/24}+886\,187\,500 \,q^{71/24} + \ldots,  \\
	7^3 g_{49} 
	&= \eta^{-1}(j^2-1489j+160\,511) \\
	&=  q^{-\frac{49}{24}}+42\,790\,636 \,q^{23/24}+40\,513\,206\,272 \,q^{47/24}+8\,543\,738\,297\,129 \,q^{71/24}+ \ldots. 
\end{align*}
We also have the relation $p(m,n)=p(n,m)$ when $m,n>0$ (see \cite[Corollary 2]{Ahlgren:2013aa}).

In order to give an arithmetic interpretation for the coefficients $p(m,n)$ for general $m$, we require a family of functions $\{P_v(\tau,s)\}_{v\in \N}$ whose first member is $P_1(\tau,s)=P(\tau,s)$.
We construct these functions in Section \ref{sec:poincare} using non-holomorphic Maass-Poincar\'e series. 
The specializations $P_v(\tau):=P_v(\tau,2)$ can be obtained by raising the elements of a certain basis $\{F_v\}_{v\in \N}$ of $M_{-2}(\Gamma_0(6),1,-1)$ to weight $0$.
The functions $F_v$ are uniquely determined by having Fourier expansion $F_v=q^{-v}+O(1)$.
They are easily constructed from $F$ and the $\Gamma_0(6)$-Hauptmodul
\[
	J_6(\tau) := \left(\frac{\eta(\tau)\eta(2\tau)}{\eta(3\tau)\eta(6\tau)}\right)^4 + \left(3\frac{\eta(3\tau)\eta(6\tau)}{\eta(\tau)\eta(2\tau)}\right)^4 = q^{-1} - 4 + 79q + 352q^2 + \ldots.
\]
For example, $F_1=F$ and
\begin{align*}
	F_2 &= F(J_6 + 14) = q^{-2} - 50 - 832\,q - 5693\,q^2 - \ldots, \\
	F_3 &= F(J_6^2 + 18J_6 + 27) = q^{-3} - 190 - 7371\,q - 108\,216\,q^2 - \ldots, \\
	F_4 &= F(J_6^3 + 22J_6^2 + 20J_6 - 1160) = q^{-4} - 370 - 48\,640\,q - 1\,100\,352\,q^2 - \ldots.
\end{align*}
We also require functions $P_{v,Q}(\tau,s)$ for each quadratic form with square discriminant. These are dampened versions of the functions $P_v(\tau,s)$, and are constructed in Section \ref{sec:quadratic}.

With $\chi_m:\sQ_{mn}^{(1)}\to\{-1,0,1\}$ as in \eqref{eq:def-chi-m} below,
we define the general twisted traces for $m,n\equiv 1\pmod{24}$ by
\begin{equation}
	\Tr_v(m,n) := 
	\begin{dcases}
		\frac{1}{\sqrt{|mn|}}\sum_{Q\in \Gamma \backslash \sQ^{(1)}_{mn}} \chi_m(Q) P_v(\tau_Q) & \text{if }mn<0, \\
		\frac{1}{2\pi} \sum_{Q\in \Gamma \backslash \sQ^{(1)}_{mn}} \chi_m(Q) \int_{C_Q} \left[\frac{\partial}{\partial s} P_v(\tau,s) \Big|_{s=2} \right] \frac{d\tau}{Q(\tau,1)} & \parbox[c]{.18\textwidth}{if $mn>0$ \\ is not a square,} \\
		\frac{1}{2\pi} \sum_{Q\in \Gamma \backslash \sQ^{(1)}_{mn}} \chi_m(Q) \int_{C_Q} \left[\frac{\partial}{\partial s} P_{v,Q}(\tau,s) \Big|_{s=2} \right] \frac{d\tau}{Q(\tau,1)} & \parbox[c]{.18\textwidth}{if $mn>0$ \\ is a square.}
	\end{dcases}
\end{equation}

The main theorem relates the coefficients $p(m,n)$ to these twisted traces, giving an arithmetic interpretation of the coefficients $p(m,n)$ for all $m,n$.

\begin{theorem} \label{thm:main}
	Suppose that $m,n \equiv 1\pmod{24}$ and that $m$ is squarefree. For each $v\geq1$ coprime to $6$ we have
	\begin{equation} \label{eq:main-thm}
		\Tr_v(m,n) = \sum_{d|v} d \pfrac{m}{v/d} \pfrac{12}{d} p(d^2 m,n).
	\end{equation}
\end{theorem}

In the proof of Theorem \ref{thm:main} we will encounter the Kloosterman sum
\begin{equation} \label{eq:def-kloo}
	K(a,b;c) := \sum_{\substack{d\bmod c \\ (d,c)=1}} e^{\pi i s(d,c)} e\left(\frac{\bar d a+db}{c}\right),
\end{equation}
where $\bar d$ denotes the inverse of $d$ modulo $c$ and $e(x):=\exp(2\pi i x)$.
Here $s(d,c)$ is the Dedekind sum
\begin{equation} \label{eq:def-dedekind-sum}
	s(d,c) := \sum_{r=1}^{c-1} \left( \frac rc - \biggl\lfloor\frac rc\biggr\rfloor - \frac 12 \right)\left( \frac{dr}c - \left\lfloor\frac{dr}c\right\rfloor - \frac 12 \right)
\end{equation}
which appears in the transformation formula for $\eta(z)$ (see Section 2.8 of \cite{Iwaniec:Topics}).
The presence of the factor $e^{\pi i s(d,c)}$ makes the Kloosterman sum quite difficult to evaluate.
The following formula, proved by Whiteman in \cite{Whiteman}, is attributed to Selberg and gives an evaluation in the special case $a=0$:
\begin{equation} \label{eq:selberg-kloo}
	K(0,b;c) = \sqrt{\frac c3} \sum_{\substack{\ell\bmod {2c} \\ (3\ell^2+\ell)/2 \equiv b (c)}} (-1)^\ell \cos\left( \frac{6\ell+1}{6c}\pi \right).
\end{equation}
Theorem \ref{thm:kloosterman} below is a generalization of \eqref{eq:selberg-kloo} which is of independent interest. For each $v$ with $(v,6)=1$ and for each $m,n\equiv 1\pmod{24}$ with $m$ squarefree, define
\begin{equation} \label{eq:def-S-v}
	S_v(m,n;24c) := \sum_{\substack{b\bmod {24c} \\ b^2\equiv mn(24c)}} \pfrac{12}{b} \chi_m\left(\left[6c,b,\tfrac{b^2-mn}{24c}\right]\right) \, e\pfrac{bv}{12c}.
\end{equation}
It is not difficult to show (see \eqref{eq:m-v-1-whiteman} below) that the right-hand side of \eqref{eq:selberg-kloo} is equal to 
\[
	\frac 14 S_1(1,24b+1;24c).
\]

\begin{theorem} \label{thm:kloosterman}
Suppose that $n=24n'+1$ and that $M=v^2m=24M'+1$, where $m$ is squarefree and $(v,6)=1$. Then
\begin{equation} \label{eq:kloo-thm}
	K(M',n';c) = 4\sqrt{\frac c3}\, \pfrac{12}v \sum_{u|(v,c)} \mu(u) \pfrac mu S_{v/u}(m,n;24c/u).
\end{equation}
\end{theorem}

When $(mn,c)=1$ we clearly have $S_v(m,n;24c) \ll_\epsilon c^\epsilon$ for any $\epsilon>0$, which yields the bound (for any fixed integers $m$, $n$)
\begin{equation} \label{eq:kloo-bound}
	K(m,n;c) \ll_\epsilon c^{\frac 12+\epsilon}, \qquad (mn,c)=1.
\end{equation}
This is reminiscent of Weil's bound (see \cite{Weil:Exponential} and \cite[Lemma 2]{Hooley:Asymptotic}) for the ordinary Kloosterman sum
\[
	k(m,n;c):=\sum_{\substack{d\bmod c \\ (d,c)=1}} e\pfrac{\bar d m+dn}{c} \ll_\epsilon (m,n,c)^{\frac 12}c^{\frac 12+\epsilon}.
\]

Our proof of Theorem \ref{thm:main} follows the method of Duke, Imamo\=glu, and T\'oth \cite{DIT:CycleIntegrals} in the case when $mn$ is not a square, and the author \cite{Andersen:Geodesics} in the case when $mn$ is a square. 
We construct the functions $P_v(\tau,s)$ and $P_{v,Q}(\tau,s)$ as Poincar\'e series and then evaluate the traces of these Poincar\'e series directly. 
We match these evaluations to the formulas given in \cite{Ahlgren:2013aa} for the coefficients $p(m,n)$. 

In Section \ref{sec:poincare} we review some facts about weak Maass forms and construct the functions $P_v(\tau,s)$. 
In Section \ref{sec:quadratic} we discuss binary quadratic forms and establish some facts which we will need for the proof of the main theorem. 
We construct the functions $P_{v,Q}(\tau,s)$ at the end of Section \ref{sec:quadratic}.
In Section \ref{sec:kloosterman} we prove a proposition which is equivalent to Theorem \ref{thm:kloosterman} and is a crucial ingredient in the proof of Theorem~\ref{thm:main}. The proof of Theorem \ref{thm:main} comprises the final section of the paper.

\section{Poincar\'e series and the construction of $P_v(\tau,s)$} \label{sec:poincare}

In this section we construct the weak Maass forms $P_v(\tau,s)$ in terms of Poincar\'e series. 
We first recall the definition and basic properties of weak Maass forms, and show how such forms can be constructed using Poincar\'e series associated with Whittaker functions.
We then discuss the Atkin-Lehner involutions $W_d$ which are used to build the functions $P_v(\tau,s)$.

\subsection{Weak Maass forms and Poincar\'e series} 
For $\gamma = \pmatrix abcd\in \GL_2^+(\Q)$ and $k\in 2\Z$, we define the weight $k$ slash operator $\sl_k$ by
\[
	\left(f \sl_k \gamma\right) (\tau) := (\det \gamma)^{k/2} (c\tau+d)^{-k} f\pfrac{a\tau+b}{c\tau+d}.
\]
Let $\Gamma'\subseteq \Gamma_1$ be a congruence subgroup.
A weak Maass form of weight $k$ and Laplace eigenvalue $\lambda$ is a smooth function $f:\H\to \C$ satisfying
\begin{enumerate}
	\item $f\sl_k \gamma=f$ for all $\gamma \in \Gamma'$,
	\item $\Delta_k f = \lambda f$, where $\Delta_k$ is defined in \eqref{eq:def-laplacian}, and
	\item $f$ has at most linear exponential growth at each cusp of $\Gamma'$. 
\end{enumerate}
If $\lambda=0$, we say that $f$ is a harmonic Maass form.
The differential operator
\[
	\xi_k := 2iy^k \frac{\bar\partial}{\partial \bar\tau}
\]
plays an important role in the theory of harmonic Maass forms.
It commutes with the slash operator; that is,
\[
	\xi_k \left(f\sl_k \gamma\right) = \left(\xi_k f\right) \sl_{2-k} \gamma.
\]
Thus, if $f$ has weight $k$ then $\xi_k f$ has weight $2-k$.
The Laplacian $\Delta_k$ decomposes as
\[
	\Delta_k = -\xi_{2-k} \circ \xi_k,
\]
which shows that $\xi_k$ maps harmonic Maass forms to weakly holomorphic modular forms.

Define
\[
	\Gamma := \Gamma_0(6) / \{\pm 1\}.
\]
We follow Section 2.6 of \cite{BO:AlgebraicFormulas} in constructing Poincar\'e series for $\Gamma$ attached to special values of the $M$-Whittaker function $M_{\mu,\nu}(y)$ (see Chapter 13 of \cite{AS:Pocketbook} for the definition and relevant properties). Let $v$ be a positive integer, and for $s\in \C$ and $y>0$, define
\begin{equation} \label{def-M-s-k}
	\M_{s,k}(y) := y^{-k/2} M_{-k/2,s-1/2}(y)
\end{equation}
and 
\[
	\phi_v(\tau,s,k):=\M_{s,k}(4\pi v y)e(-vx).
\]
Then 
\begin{equation} \label{eq:phi-bound}
	\phi_v(\tau,s,k)\ll y^{\re(s)-k/2} \text{ as } y\to 0.
\end{equation}
Letting $\Gamma_\infty:=\{\pmatrix 1*01\}\subseteq \Gamma$ denote the stabilizer of $\infty$, we define the Poincar\'e series
\begin{equation}
	\F_v(\tau,s,k) := \frac{1}{\Gamma(2s)} \sum_{\gamma\in\Gamma_\infty\backslash\Gamma} (\phi_v\sl_k \gamma)(\tau,s,k).
\end{equation}
On compact subsets of $\H$, we have (by \eqref{eq:phi-bound}) the bound
\[
	|\F_v(\tau,s,k)| \ll y^{\re(s)-k/2}\sum_{\pmatrix abcd \in \Gamma_{\infty} \backslash\Gamma} |c\tau+d|^{-2\re (s)},
\]
so $\F_v(\tau,s,k)$ converges normally for $\re(s)>1$.
A computation involving (13.1.31) of \cite{AS:Pocketbook} shows that
\[
	\Delta_k \, \phi_v(\tau,s,k) = (s-k/2)(1-k/2-s) \phi_v(\tau,s,k).
\]
Since $\F_v$ clearly satisfies $\F_v \sl_k \gamma = \F_v$ for all $\gamma\in \Gamma$, we see that for fixed $s$ with $\re(s)>1$, the function $\F_v(\tau,s,k)$ is a weak Maass form of weight $k$ and Laplace eigenvalue $(s-k/2)(1-k/2-s)$.

We are primarily interested in the case when $k$ is negative.
In this case the special value $\F_v(\tau,1-k/2,k)$ is a harmonic Maass form.
Its principal part at $\infty$ is given by $q^{-v}+c_0$ for some $c_0\in \C$, while its principal parts at the other cusps are constant.
Thus, $\xi_k \F_v(\tau,1-k/2,k)$ is a cusp form of weight $2-k$ on $\Gamma$.

\subsection{Atkin-Lehner involutions} 
We recall some basic facts on Atkin-Lehner involutions (see, for example, Section IX.7 of \cite{Knapp:elliptic} or Section 2.4 of \cite{Ono:Web}).
Suppose that $N$ is a positive squarefree integer and that $d \mid N$. 
Let $W_d=W_d^N$ denote any matrix with determinant $d$ of the form
\[
	W_d = \pMatrix{d\alpha}{\beta}{N\gamma}{d\delta}
\]
with $\alpha,\beta,\gamma,\delta\in \Z$.
The relation
\begin{equation} \label{eq:W-d-normalizes}
	W_d \ \Gamma_0(N) \ W_d^{-1} = \Gamma_0(N)
\end{equation}
shows that the map $f\mapsto f\sl_k W_d$ (called the Atkin-Lehner involution $W_d$) is independent of the choices of $\alpha,\beta,\gamma,\delta$ and defines an involution on the space of weight $k$ forms on $\Gamma_0(N)$.
If $d$ and $d'$ are divisors of $N$, then
\[
	f \sl_k W_d \sl_k W_{d'} = f \sl_k W_{d*d'},
\]
where $d*d'=dd'/(d,d')^2$.
When $d=N$ it is convenient to take $W_N=\pmatrix0{-1}N0$. Finally, the Atkin-Lehner involutions act transitively on the cusps of $\Gamma_0(N)$; that is, for each cusp $\mathfrak a \in \Gamma_0(N)\backslash \P^1(\Q)$, there exists a unique $d\mid N$ such that $W_d\,\infty=\mathfrak a$.

Recall that $F_v(\tau)=q^{-v}+O(1)$ is a weakly holomorphic modular form of weight $-2$ on $\Gamma_0(6)$ with eigenvalues $1$ and $-1$ under $W_6$ and $W_2$, respectively.
We claim that
\begin{equation} \label{eq:F-v-poincare}
	F_v(\tau) = \sum_{d|6} \mu(d) \left(\F_v \sl_{-2} W_d\right)(\tau,2,-2).
\end{equation}
To prove this, let $\tilde F_v(\tau)$ denote the right-hand side of \eqref{eq:F-v-poincare}.
Since $\xi_{-2}\F_v(\tau,2,-2)$ lies in the one-dimensional space of cusp forms of weight $4$ on $\Gamma_0(6)$, it must be proportional to
\begin{equation*} \label{eq:G-v-shadow}
	g(\tau) := \left( \eta(\tau)\eta(2\tau)\eta(3\tau)\eta(6\tau) \right)^2.
\end{equation*}
Since $\xi_k$ commutes with the slash operator and $g(\tau)$ is invariant under $W_d$ for each $d\mid 6$, we have, for some $\alpha\in\C$, the relation
\[
	\xi_{-2} \tilde F_v(\tau) = \alpha\sum_{d\mid 6} \mu(d) g(\tau) = 0.
\]
Thus $F_v(\tau)-\tilde F_v(\tau)$ is holomorphic on $\H$ and vanishes at every cusp. 
Hence $F_v(\tau)=\tilde F_v(\tau)$.

\subsection{The functions $P_v(\tau)$, $P_v(\tau,s)$} 
To construct the functions $P_v(\tau)$ and $P_v(\tau,s)$, we require the Maass raising operator
\begin{equation} \label{eq:def-raising}
	R_k := 2i\frac{\partial}{\partial \tau} + \frac ky
\end{equation}
which raises the weight of a weak Maass form by $2$.
For each $v\geq 1$, we define
\[
	P_v(\tau) := \frac{1}{4\pi v} R_{-2} F_v(\tau).
\]
Then $P_v(\tau)$ is a weak Maass form of weight $0$ and Laplace eigenvalue $-2$.
By \eqref{eq:F-v-poincare} and Proposition 2.2 of \cite{BO:AlgebraicFormulas} this is equivalent to defining
\begin{align*}
	P_v(\tau) :=& \sum_{d|6} \mu(d) \F_v (W_d\,\tau,2,0) \\
	=& \frac{1}{6}\sum_{d\mid 6} \mu(d) \sum_{\gamma\in \Gamma_\infty \backslash \Gamma} \phi_v\left(\gamma W_d \, \tau,2,0\right).
\end{align*}
Similarly, we define $P_v(\tau,s)$ as
\begin{align}
	P_v(\tau,s) :=& C(s) \sum_{d|6} \mu(d) \F_v (W_d\,\tau,s,0) \notag \\
	=& \frac{C(s)}{\Gamma(2s)}\sum_{d\mid 6} \mu(d) \sum_{\gamma\in \Gamma_\infty \backslash \Gamma} \phi_v\left(\gamma W_d \, \tau,s,0\right), \label{eq:def-p-v-tau-s}
\end{align}
where
\begin{equation}
	C(s) := \frac{2^s}{\pi} \Gamma\pfrac{s+1}{2}^2.
\end{equation}
We have chosen the non-standard normalizing factor $C(s)$ so that later results are cleaner to state. Note that $C(2)=1$, so
\begin{equation} \label{eq:P-v-tau-2=P-v-tau}
	P_v(\tau,2)=P_v(\tau).
\end{equation}
In the next section we will define the dampened functions $P_{v,Q}(\tau,s)$.

\section{Binary quadratic forms and the functions $P_{v,Q}(\tau,s)$} \label{sec:quadratic}

In this section we recall some basic facts about binary quadratic forms and the genus characters $\chi_m$. A good reference for this material is Section I of \cite{GKZ:Heegner}. Throughout this section, we assume that $m,n\equiv 1\pmod{24}$ and that $m$ is squarefree. The latter condition ensures that $m$ is a fundamental discriminant.

Suppose that $r\in\{1,5,7,11\}$. 
We recall that
\[
	\sQ_n^{(r)} := \left\{ ax^2+bxy+cy^2: b^2-4ac=n, \ 6\mid a, \ b\equiv r\pmod*{12}, \text{ and }a>0 \text{ if }n<0 \right\}.
\] 
Let $\Gamma^*$ denote the group generated by $\Gamma=\Gamma_0(6)/\{\pm 1\}$ and the Atkin-Lehner involutions $W_d$ for $d\mid 6$.
Matrices $\gamma=\pmatrix ABCD\in\Gamma^*$ act on such forms on the left by
\[
	\gamma Q(x,y) := \frac{1}{\det \gamma}Q(Dx-By,-Cx+Ay).
\]
It is easy to check that this action is compatible with the action $\gamma\tau:=\frac{A\tau+B}{C\tau+D}$ on the roots of $Q$:
for all $\gamma\in \Gamma^*$, we have
\begin{equation} \label{eq:root-compatible}
	\gamma\,\tau_Q = \tau_{\gamma Q}.
\end{equation}
The set $\Gamma \backslash \sQ_n^{(r)}$ forms a finite group under Gaussian composition which is isomorphic to the narrow class group of $\Q(\sqrt{n})/\Q$ when $n$ is a fundamental discriminant.
Let $\sQ_n$ denote the union
\[
	\sQ_n := \bigcup_{r\in \{1,5,7,11\}} \sQ_n^{(r)}.
\]

For $d\mid 6$, the Atkin-Lehner involution $W_d=\pmatrix{d\alpha}{\beta}{6\gamma}{d\delta}$ acts on quadratic forms by
\begin{equation} \label{eq:atkin-lehner-q}
	W_d \, Q(x,y) := \frac 1d \, Q(d\delta x-\beta y,-6\gamma x+d\alpha y).
\end{equation}
A computation involving \eqref{eq:atkin-lehner-q} and the relation $d\alpha\delta - \frac 6d \beta\gamma=1$ shows that
\begin{equation} \label{atkin-lehner-q-2}
	W_d \, [6a,b,c] = \left[6*, b\left(1+\tfrac {12}d \beta\gamma\right)+12*,*\right].
\end{equation}
It is convenient to choose $W_2=\pmatrix{2}{-1}{6}{-2}$ and $W_3=\pmatrix{3}{1}{6}{3}$. Then \eqref{atkin-lehner-q-2} shows that
\begin{equation} \label{eq:W-d-r-r'}
	W_d : \sQ_n^{(r)} \longleftrightarrow \sQ_n^{(r')}
\end{equation}
is a bijection, where
\begin{equation} \label{eq:r'-r-cases}
	r'\equiv r\times
	\begin{cases}
		1 & \text{ if } d=1, \\
		7 & \text{ if } d=2, \\
		5 & \text{ if } d=3, \\
		11 & \text{ if } d=6,
	\end{cases}
	\pmod{12}.
\end{equation}
Moreover, we have
\begin{equation} \label{eq:Qn-decomp-W-d}
	\sQ_n = \bigcup_{d|6} W_d \, \sQ_n^{(r)}
\end{equation}
for any $r\in \{1,5,7,11\}$.

We turn now to the extended genus character $\chi_m$. For $Q\in \sQ_{mn}$, define
\begin{equation} \label{eq:def-chi-m}
	\chi_m(Q) := 
	\begin{cases}
		\pfrac{m}{r} & \text{ if $(a,b,c,m)=1$ and $Q$ represents $r$ with $(r,m)=1$}, \\
		0 & \text{ if $(a,b,c,m)>1$}.
	\end{cases}
\end{equation}
The following lemma lists some properties of $\chi_m$.
\begin{lemma} \label{lem:chi-props}
Suppose that $m,n\equiv 1\pmod{24}$ and that $m$ is squarefree.
\begin{enumerate}[P1\textup{)}]
	\item The map $\chi_m:\Gamma \backslash \sQ_{mn}\to\{-1,0,1\}$ is well-defined; 
	i.e. $\chi_m(\gamma Q)=\chi_m(Q)$ for all $\gamma\in \Gamma$.
	\item If $(a,a')=1$ then
	\[
		\chi_m([6aa',b,c]) = \chi_m([6a,b,a'c])\chi_m([6a',b,ac]).
	\]
	\item For each $d\mid 6$ we have
	\[
		\chi_m(Q) = \chi_m(W_d\,Q).
	\]
	\item Suppose that $[6a,b,c]\in \sQ_{mn}$, and let $g:=\pm(a,m)$, where the sign is chosen so that $g\equiv 1\pmod4$. Then
	\[
		\chi_m([6a,b,c]) = \pfrac{m/g}{6a}\pfrac{g}{c}.
	\]
	\item We have $\chi_m(-Q)=\sgn(m) \chi_m(Q)$.
\end{enumerate}
\end{lemma}
\begin{proof}
Property P3 for $d=1,6$ and P1, P2, and P4 are special cases of Proposition 1 of \cite{GKZ:Heegner}. Property P5 follows easily from P4. A generalization of P3 is stated without proof in \cite{GKZ:Heegner}, so we provide a proof here for our special case.

We want to show that P3 holds for $d=2,3$. Suppose that
\[
	 Q=[6a,b,c] \text{ with } (a,b,c,m)=1. 
\]
Choosing $W_2=\pmatrix 2{-1}6{-2}$ and $W_3=\pmatrix 3163$, we find that
\begin{gather*}
	W_2\,[6a,b,c] = [6(2a+b+3c),*,3a+b+2c], \\
	W_3\,[6a,b,c] = [6(3a-b+2c),*,2a-b+3c].
\end{gather*}
We will use P4. For $W_2$, we want to show that
\begin{equation} \label{eq:W-2-want}
	\pfrac{m/g}{6(2a+b+3c)}\pfrac{g}{3a+b+2c} = \pfrac{m/g}{6a}\pfrac{g}{c}.
\end{equation}
Since $g$ divides $mn+4ac=b^2$ and $g$ is squarefree, we see that $g\mid b$, so
\begin{equation} \label{eq:g-3a-b-2c}
	\pfrac{g}{3a+b+2c} = \pfrac{g}{c}\pfrac{g}{2}.
\end{equation}
If $a=0$ then $g=m$, and \eqref{eq:W-2-want} follows from \eqref{eq:g-3a-b-2c} and the fact that $m\equiv 1\pmod 8$.
If $a\neq 0$, then the relation
\[
	8a(2a+b+3c)=(4a+b)^2-m
\]
shows that
\[
	\pfrac{m/g}{6(2a+b+3c)}\pfrac{m/g}{6a}=\pfrac{m/g}{2}.
\]
Together with \eqref{eq:g-3a-b-2c}, this completes the proof for $W_2$.
The proof for $W_3$ is similar.
\end{proof}

The remainder of this section follows Sections 3 and 4 of \cite{DIT:CycleIntegrals} and Section 3 of \cite{Andersen:Geodesics}.
Let $\Gamma_Q$ denote the stabilizer of $Q$ in $\Gamma=\Gamma_0(6)/\{\pm1\}$. 
When the discriminant of $Q$ is negative or a positive square, the group $\Gamma_Q$ is trivial. 
However, when the discriminant $n>0$ is not a square, the group $\Gamma_Q$ is infinite cyclic. If $Q=[a,b,c] \in \sQ_n$ with $(a,b,c)=1$, we have $\Gamma_Q=\langle g_Q\rangle$, where
\[
	g_Q := \pMatrix{\frac{t+bu}2}{cu}{-au}{\frac{t-bu}2}
\]
and $t,u$ are the smallest positive integral solutions to Pell's equation $t^2-nu^2=4$. 
When $(a,b,c)=\delta>1$, we have $\Gamma_Q=\langle g_{Q/\delta} \rangle$.

For $Q=[a,b,c]\in \sQ_n$ with $n>0$, let $S_Q$ denote the geodesic in $\H$ connecting the two roots of $Q(\tau,1)$.
Explicitly, $S_Q$ is the curve in $\H$ defined by
\[
	a|\tau|^2 + b\re(\tau) + c = 0.
\]
When $a\neq 0$, $S_Q$ is a semicircle, which we orient counter-clockwise if $a>0$ and clockwise if $a<0$.
When $a=0$, $S_Q$ is the vertical line $\re(\tau)=-c/b$, which we orient upward.
If $\gamma\in \Gamma$ then we have
\begin{equation} \label{eq:gamma-S-Q}
	\gamma S_Q=S_{\gamma Q}
\end{equation}
Fix any $z\in S_Q$ and define the cycle $C_Q$ as the directed arc on $S_Q$ from $z$ to $g_Q z$.
We define
\begin{equation} \label{eq:d-tau-q-def}
	d\tau_Q := \frac{\sqrt{n}\, d\tau}{Q(\tau,1)},
\end{equation}
so that if $\tau'=\gamma\tau$ for some $\gamma\in \Gamma^*$, we have
\begin{equation} \label{eq:d-tau-prime}
	d\tau'_{\gamma Q} = d\tau_Q.
\end{equation}

Suppose that $Q$ has positive non-square discriminant and that $f$ is a $\Gamma$-invariant function that is continuous on $S_Q$.
A straightforward generalization of Lemma 6 of \cite{DIT:CycleIntegrals} shows that the integral
\[
	\int_{C_Q} f(\tau) d\tau_Q
\]
is a well-defined (i.e. independent of the choice of $z\in S_Q$) invariant of the equivalence class of $Q$.

We now define the functions $P_{v,Q}(\tau,s)$. Let $Q=[a,b,c]$ be a binary quadratic form with square discriminant.
Then the equation $Q(x,y)=0$ has two inequivalent solutions $[r_1:s_1]$ and $[r_2:s_2]$ in $\P^1(\Q)$, which we write as fractions $\mathfrak{a}_i:=r_i/s_i$, with $(r_i,s_i)=1$ and possibly $s_i=0$. 
For each $i$, there is a unique $d_i$ such that
\[
	W_{d_i}\, \mathfrak a_i \sim_\Gamma \infty.
\]
Thus, up to translation, there is a unique $\gamma_i\in \Gamma$ such that
\[
	\gamma_i W_{d_i}\, \mathfrak a_i = \infty.
\]
The function $P_{v,Q}(\tau,s)$ is defined by deleting the two terms of $P_v(\tau,s)$ in \eqref{eq:def-p-v-tau-s} corresponding to the pairs $(\gamma_i,W_{d_i})$; that is,
\begin{equation} \label{eq:def-p-v-q}
	P_{v,Q}(\tau,s) := \frac{C(s)}{\Gamma(2s)}\sum_{d|6} \mu(d) \sum_{\substack{\gamma\in \Gamma_\infty\backslash \Gamma \\ \gamma W_d\, \mathfrak{a}_i \neq \infty}} \phi_v(\gamma W_d\,\tau,s,0).
\end{equation}
Suppose that $\sigma\in \Gamma$. Then by \eqref{eq:root-compatible}, the roots of $\sigma Q$ are $\sigma \mathfrak a_1$ and $\sigma \mathfrak a_2$, and we find (using \eqref{eq:W-d-normalizes}) that
\[
	P_{v,\sigma Q}(\sigma\tau,s) = P_{v,Q}(\tau,s).
\]
Together with \eqref{eq:gamma-S-Q} and \eqref{eq:d-tau-prime}, this shows that the integral
\[
	\int_{C_Q} P_{v,Q}(\tau,s) \frac{d\tau}{Q(\tau,1)}
\]
is invariant under $Q\mapsto\sigma Q$.

\section{Kloosterman sums and the proof of Theorem \ref{thm:kloosterman}} \label{sec:kloosterman}

In this section we prove an identity (Proposition \ref{prop:s-k} below) connecting the Kloosterman sum \eqref{eq:def-kloo} with the twisted quadratic Weyl sum \eqref{eq:def-S-v}. 
This is an essential ingredient in the proof of Theorem \ref{thm:main}, and is equivalent to the evaluation of the Kloosterman sum in Theorem~\ref{thm:kloosterman}.

Throughout this section, $v$ is a positive integer coprime to $6$ and $m,n\equiv 1\pmod{24}$ with $m$ squarefree.
We will use the notation
\[
	a' := \frac{a-1}{24}
\]
whenever $a\equiv 1\pmod{24}$.
The Kloosterman sum is defined as
\[
	K(a,b;c) := \sum_{d(c)^*} e^{\pi i s(d,c)} e\pfrac{\bar da + db}{c},
\]
where $d(c)^*$ indicates that the sum is taken over residue classes coprime to $c$, and $\bar d$ denotes the inverse of $d$ modulo $c$.
The factor $e^{\pi i s(d,c)}$ makes the Kloosterman sum very difficult to evaluate.
The following lemma shows that $e^{\pi i s(d,c)}$ is related to the Gauss-type sums
\begin{equation} \label{eq:def-H-d-c}
	H_{d,c}(\delta) := \frac 12 \sum_{j(2c)} e\pfrac{d(6j+\delta)^2}{24c}
\end{equation}
which were introduced by Fischer in \cite{Fischer:Dedekind}.

\begin{lemma} \label{lem:s-d-c-H}
Suppose that $(v,6)=(c,d)=1$ and define
\begin{align*}
	\alpha &:= 1-\bar{d}c-\bar{d}v, \\
	\beta &:= 1-\bar{d}c+\bar{d}v,
\end{align*}
with $\bar d$ chosen such that
\begin{equation*}
	d \bar d \equiv 
	\begin{cases}
		1\pmod{c} & \text{ if $c$ is odd},\\
		1\pmod{2c} & \text{ if $c$ is even}.
	\end{cases}
\end{equation*}
Then we have
\begin{equation} \label{eq:s-d-c-H}
	\sqrt{3c} \pfrac{12}{v} e\pfrac{\bar{d}(v^2-1)}{24c} e^{\pi i\,s(d,c)} 
	= e\pfrac{2v+d\alpha^2}{24c} H_{-d,c}(\alpha) + e\pfrac{-2v+d\beta^2}{24c} H_{-d,c}(\beta).
\end{equation}
\end{lemma}

\begin{proof}
Define
\[
	f(v):= \pfrac{12}{v} e\pfrac{2v+d\alpha^2-\bar{d}(v^2-1)}{24c} H_{-d,c}(\alpha).
\]
We will prove \eqref{eq:s-d-c-H} by showing that $f(v)+f(-v)=\sqrt{3c}\,e^{\pi i\,s(d,c)}$.

We first show that for fixed $d,c$, the function $f(v)$ depends only on $v\bmod 6$. By (3.8) of \cite{Whiteman} we find that $H_{-d,c}(\alpha)$ depends only on $\alpha\bmod 6$. Define $\ep\in \{-1,1\}$ by $v\equiv \ep \pmod{6}$. Then $\pfrac{12}{v}=e\pfrac{v-\ep}{12}$, and we have
\begin{multline*}
	\pfrac{12}{v}e\pfrac{2v+d\alpha^2-\bar{d}(v^2-1)}{24c} \\*
	= e\left( v^2 \frac{\bar{d}(d\bar d-1)/c}{24} - v\frac{(d\bar d-1)/c-d\bar d^2-1}{12} - \frac{\ep}{12} + \frac{d(\bar d c-1)^2+\bar d}{24c} \right).
\end{multline*}
This depends only on $v\bmod 6$ since $v^2\equiv 1\pmod{24}$ and, by definition,
\[
	\frac{d\bar d-1}c-d\bar d^2-1 \in 2\Z.
\]

Now $f(v)+f(-v)=f(\ep)+f(-\ep)$ is independent of $v$, so to prove \eqref{eq:s-d-c-H} it suffices to show that $f(1)+f(-1)=\sqrt{3c}\,e^{\pi i\,s(d,c)}$. This is proved in Section 4 of \cite{Whiteman}.
\end{proof}

The quadratic Weyl sum $S_v(m,n;24c)$ is defined as
\[
	S_v(m,n;24c) := \sum_{\substack{b(24c)\\b^2\equiv mn(24c)}} \chi_{12}(b) \chi_m\left([6c,b,\tfrac{b^2-mn}{24c}]\right) e\pfrac{bv}{12c},
\]
where $\chi_m$ is defined in \eqref{eq:def-chi-m}.
We clearly have $S_v(m,n;24c) = \overline{S_v(m,n;24c)}$, so the exponential $e\pfrac{bv}{12c}$ may be replaced by $\cos\pfrac{bv\pi}{6c}$.
When $m=1$, we obtain a simpler expression for $S_1(1,n;24c)$ as follows. The summands of $\re(S_v(1,n;24c))$ are invariant under both $b\mapsto b+12c$ and $b\mapsto -b$, so we may sum over those $b$ modulo $12c$ for which $b\equiv 1\pmod{6}$, and multiply the sum by $4$. Writing $b=6\ell+1$, we obtain (cf. formula \eqref{eq:selberg-kloo})
\begin{equation} \label{eq:m-v-1-whiteman}
	S_1(1,n;24c) = 4\sum_{\substack{\ell\bmod 2c\\(3\ell^2+\ell)/2\equiv n'(c)}} (-1)^\ell \cos\pfrac{(6\ell+1)v\pi}{6c}.
\end{equation}

The following proposition gives an expression for $S_v(m,n;24c)$ in terms of Kloosterman sums.
Its proof occupies the remainder of the section.
Theorem~\ref{thm:kloosterman} follows from \eqref{eq:s-k} by M\"obius inversion in two variables.

\begin{proposition} \label{prop:s-k}
Suppose that $m,n\equiv 1\pmod{24}$ and that $m$ is squarefree. Suppose that $c,v>0$ and that $(v,6)=1$. Then
\begin{equation} \label{eq:s-k}
	S_v(m,n;24c) = 4\sqrt{3} \sum_{u|(v,c)} \pfrac{12}{v/u}\pfrac mu \sqrt{\frac uc} \, K\left( \left(\tfrac{v^2}{u^2}\,m\right)', n'; \frac cu \right).
\end{equation}
\end{proposition}

\begin{remark}
Proposition \ref{prop:s-k} resembles Proposition 3 of 
\cite{DIT:CycleIntegrals}, which is proved using a slight modification of Kohnen's argument in \cite[Proposition 5]{Kohnen:Fourier}.
Using an elegant idea of T\'oth \cite{Toth:Salie}, Duke \cite{Duke:Uniform} greatly simplified Kohnen's proof for the case $m=D=1$ (in the notation of \cite{DIT:CycleIntegrals}). 
Jenkins \cite{Jenkins:Kloosterman} later extended this argument to the case of general $m$.
However, Kohnen's argument remains the only proof of the general case.

Although special cases of \eqref{eq:s-k} are amenable to the methods of Duke and Jenkins, we prove Proposition \ref{prop:s-k} in full generality by adapting Kohnen's argument.
The proof is quite technical.
\end{remark}

In the proof of Proposition \ref{prop:s-k} we will encounter the quadratic Gauss sum
\begin{equation} \label{eq:def-gauss}
	G(a,b,c) := \sum_{x(c)} e\pfrac{ax^2+bx}{c}, \qquad c>0.
\end{equation}
For any $d\mid (a,c)$, we see by replacing $x$ by $x+c/d$ that
\begin{equation}
	G(a,b,c) = e\ptfrac{b}{d} G(a,b,c).
\end{equation}
This implies that $G(a,b,c)=0$ unless $d\mid b$. 
In that case,
\begin{equation} \label{eq:gauss-gcd}
	G(a,b,c) = d \cdot G\left(\frac ad, \frac bd, \frac cd\right).
\end{equation}
If $(a,c)=1$, we have the well-known evaluations (see Theorems 1.5.1, 1.5.2, and 1.5.4 of \cite{BEW:Gauss})
\begin{equation} \label{eq:gauss-eval}
	G(a,0,c) = 
	\begin{cases}
		0 & \text{ if } 2\mid\mid c, \\
		(1+i)\ep_a^{-1} \sqrt{c} \pfrac ca & \text{ if } 4\mid c, \\
		\ep_c \sqrt{c} \pfrac ac & \text{ if $c$ is odd,}
	\end{cases}
\end{equation}
where
\[
	\ep_a :=
	\begin{cases}
		1 & \text{ if } a\equiv 1\pmod{4},\\
		i & \text{ if } a\equiv 3\pmod{4}.
	\end{cases}
\]
If $4\mid c$ and $(a,c)=1$ then, by replacing $x$ by $x+c/2$, we find that $G(a,b,c)=0$ if $b$ is odd. 
If $b$ is even and $4\mid c$, or if $c$ is odd, then by completing the square and using \eqref{eq:gauss-eval}, we find that
\begin{equation} \label{eq:gauss-4c-eval}
 	G(a,b,c) = 
 	\begin{dcases}
 		e\pfrac{-\bar a\,b^2}{4c} (1+i) \ep_a^{-1} \sqrt c \pfrac ca & \text{ if $b$ is even and $4\mid c$}, \\
 		e\pfrac{-\bar{4a}\,b^2}{c} \ep_c \sqrt c \pfrac ac & \text{ if $c$ is odd.}
 	\end{dcases}
\end{equation} 
Finally, these Gauss sums satisfy the multiplicative property
\begin{equation} \label{eq:gauss-mult}
	G(a,b,qr) = G(ar,b,q)\,G(aq,b,r), \qquad (q,r)=1
\end{equation}
which is a straightforward generalization of \cite[Lemma 1.2.5]{BEW:Gauss}.

We will need an explicit formula for $\chi_m([6c,b,\frac{b^2-mn}{24c}])$, which follows from P4 of Lemma \ref{lem:chi-props} (see also Proposition 6 of \cite{Kohnen:Fourier}). For each odd prime $p$, let
\[
	p^*:=(-1)^{\frac{p-1}{2}}p
\]
so that $(\frac{a}{p})=(\frac{p^*}{a})$.
If $m$ is squarefree, then
\begin{equation} \label{eq:chi-explicit}
	\chi_m([6c,b,\tfrac{b^2-mn}{24c}]) = \prod_{\substack{p^\lambda || c \\ p\nmid m}} \pfrac{m}{p^\lambda} \prod_{\substack{p^\lambda || c \\ p|m}} \pfrac{m/p^*}{p^\lambda} \pfrac{p^*}{(b^2-mn)/p^\lambda}.
\end{equation}

\begin{proof}[Proof of Proposition \ref{prop:s-k}]
Both sides of \eqref{eq:s-k} are periodic in $v$ with period $12c$, so it suffices to show that their Fourier transforms are equal. For each $h\in \Z$ we will show that
\begin{multline} \label{eq:fourier-transforms}
	\frac{1}{12c} \sum_{v(12c)} e\pfrac{-hv}{12c} S_v(m,n;24c) \\*
	= \frac{4\sqrt 3}{12c} \sum_{v(12c)} e\pfrac{-hv}{12c} \sum_{u|(c,v)} \pfrac{12}{v/u}\pfrac mu \sqrt{\frac uc} \, K\left( \left(\tfrac{v^2}{u^2}\,m\right)', n'; \frac cu \right).
\end{multline}
Let $L(h)$ and $R(h)$ denote the left- and right-hand sides of \eqref{eq:fourier-transforms}, respectively. Then we have
\begin{align}
	L(h) 
	&= \sum_{b^2\equiv mn(24c)} \chi_{12}(b) \chi_m \left([6c,b,\tfrac{b^2-mn}{24c}]\right) \times \frac{1}{12c} \sum_{v(12c)} e\pfrac{(b-h)v}{12c} \notag\\
	&= 	\begin{cases}
		2\,\chi_{12}(h)\,\chi_m\left([6c,h,\tfrac{h^2-mn}{24c}]\right) &\text{ if }h^2\equiv mn \pmod{24c},\\
		0 & \text{ otherwise}.
		\end{cases} \label{eq:L-h-cases}
\end{align}

For the right-hand side, we have
\begin{align*}
	R(h)
	&= \frac{1}{\sqrt3\,c} \sum_{u|c} \pfrac mu \sqrt{\frac uc} \sum_{v(12c/u)} \pfrac{12}{v} e\pfrac{-hv}{12 c/u} \sum_{d(c/u)^*} e^{\pi i\,s(d,c/u)} e\pfrac{\bar d(v^2 m)'+dn'}{c/u} \\
	&= \frac{1}{\sqrt3\,c} \sum_{u|c} \pfrac m{c/u} \frac{1}{\sqrt u} \sum_{d(u)^*} e\pfrac{dn'}{u} \sum_{v(12u)} e^{\pi i\,s(d,u)} e\pfrac{-\bar d}{24u} \pfrac{12}{v} e\pfrac{\bar dmv^2-2hv}{24u}.
\end{align*}
Using Lemma \ref{lem:s-d-c-H}, \eqref{eq:def-H-d-c}, and the fact that $u+v$ is even, we find that
\begin{multline*}
	e^{\pi i\,s(d,u)} e\pfrac{-\bar d}{24u}\pfrac{12}{v}  \\
	= \frac{1}{2\sqrt{3u}} \, e\pfrac{-\bar dv^2}{24u} \sum_{j(2u)} e\left( \frac{-d(3j^2+j)/2}{u} + \frac{j}{2} \right) \left( e\pfrac{v(6j+1)}{12u} + e\pfrac{-v(6j+1)}{12u} \right).
\end{multline*}
Thus we obtain
\begin{multline} \label{eq:r-h-2-gauss-sums}
	R(h) = \frac{1}{6c} \sum_{u|c} \pfrac{m}{c/u} u^{-1} \sum_{d(u)^*} e\pfrac{dn'}{u} \sum_{j(2u)} e\left( -\frac{d(3j^2+j)/2}{u} + \frac j2 \right) \\
	\times \left( G(\bar d(m-1)/2,6j+1-h,12u) + G(\bar d(m-1)/2,-6j-1-h,12u) \right),
\end{multline}
where $G(a,b,c)$ is the quadratic Gauss sum defined in \eqref{eq:def-gauss}. Since $12\mid(m-1)/2$, we see by \eqref{eq:gauss-gcd} that
\begin{multline*}
	G(d(m-1)/2,\pm(6j+1)-h,12u) \\*
	=\begin{cases}
		12 \, G\left(\bar dm', \frac{\pm(6j+1)-h}{12},u\right) & \text{ if } h\equiv \pm (6j+1) \pmod{12}, \\
		0 & \text{ otherwise. }
	\end{cases}
\end{multline*}
In particular, $R(h)=0$ unless $h\equiv \pm 1\pmod 6$.

For the remainder of the proof, we assume that $h\equiv 1\pmod{6}$ (the other case is analogous). Then in \eqref{eq:r-h-2-gauss-sums} the second Gauss sum is zero, and we take only those $j$ for which $j\equiv \frac{h-1}{6} \pmod{2}$. We write $j=2k+\frac{h-1}{6}$; then
\begin{multline*}
	\sum_{j (2u)} e\left( -\frac{d(3j^2+j)/2}{u} + \frac j2 \right) G(\bar d(m-1)/2,6j+1-h,12u) \\
	= 12 e\left(\frac{h-1}{12} - \frac{d(h^2-1)/24}{u}\right) \sum_{k(u)} e\left(-\frac{d k(6k+h)}{u}\right) G(\bar dm',k,u).
\end{multline*}
Since $e(\frac{h-1}{12})=\chi_{12}(h)$ we have
\begin{equation} \label{eq:R-h-F-h}
	R(h) = \frac 2c \, \chi_{12}(h) \sum_{u|c} \pfrac{m}{c/u} u^{-1} F_h(u),
\end{equation}
where
\[
	F_h(u) := \sum_{d(u)^*} e\left(\frac{d}{u}\left(n'-\frac{h^2-1}{24}\right)\right) \sum_{k(u)} e\pfrac{-dk(6k+h)}{u} G(\bar dm',k,u).
\]

We show that $F_h(u)$ is multiplicative as a function of $u$. 
To prove this, suppose that $(q,r)=1$, and choose $\bar r$ and $\bar q$ such that $r\bar r+q\bar q=1$. 
Let $\alpha:=n'-(h^2-1)/24$, and write $d=r\bar rx+q\bar qy$ and $k=k_1r+k_2q$. Using \eqref{eq:def-gauss} and \eqref{eq:gauss-mult} we find that that
\begin{multline*}
	F_h(qr) = \sum_{x(q)^*} e\pfrac{\bar r x\alpha}{q} \sum_{k_1(q)} e\pfrac{-k_1x(6k_1r+h)}{q}G(r\bar xm',k_1r,q) \\
	\times \sum_{y(r)^*} e\pfrac{\bar q y\alpha}{r} \sum_{k_2(r)} e\pfrac{-k_2y(6k_2q+h)}{r}G(q\bar ym',k_2q,r).
\end{multline*}
Replacing $k_1$, $k_2$, $x$, and $y$ by $k_1\bar r$, $k_2\bar q$, $xr$, and $yq$, respectively we conclude that
\[
	F_h(qr) = F_h(q)F_h(r).
\]

Clearly $\frac 1c \sum_{u|c} \pfrac{m}{c/u} u^{-1} F_h(u)$ is multiplicative as a function of $c$. 
Thus, by \eqref{eq:chi-explicit}, \eqref{eq:L-h-cases}, and \eqref{eq:R-h-F-h}, to show that $L(h)=R(h)$ it suffices to show that for each prime power $p^\lambda \mid\mid c$ we have
\begin{equation} \label{eq:F-h-c-prime-power}
	p^{-\lambda} \sum_{j=0}^\lambda \pfrac{m}{p^{\lambda-j}} p^{-j} F_h(p^j) = 
	\begin{dcases}
		\pfrac{m}{p^\lambda} &\text{ if }p\nmid m, \\
		\pfrac{m/p^*}{p^\lambda}\pfrac{p^*}{(h^2-mn)/p^\lambda} &\text{ if }p\mid m,
	\end{dcases}
\end{equation}
when $h^2\equiv mn\pmod{24p^\lambda}$, and $0$ otherwise.
 
Suppose first that $p$ is an odd prime. Set 
\[
	p^\mu:=(m',p^j).
\] 
Then $G(\bar dm',k,p^j)=0$ unless $p^\mu \mid k$. In the latter case, using \eqref{eq:gauss-gcd} and \eqref{eq:gauss-4c-eval}, we find that
\[
	G(\bar dm',k,p^j) = \ep_{p^{j-\mu}} p^{\frac{j+\mu}{2}} \pfrac{\bar d m'/p^\mu}{p^{j-\mu}} e\pfrac{-d\overline{(4m'/p^\mu)}(k/p^\mu)^2}{p^{j-\mu}}.
\]
Writing $k=p^\mu \ell$, we find that
\begin{multline} \label{eq:F-h-p-j-two-sums}
	F_h(p^j) = \ep_{p^{j-\mu}} \pfrac{m'/p^\mu}{p^{j-\mu}} p^{\frac{j+\mu}{2}} \sum_{d(p^j)^*} \pfrac{d}{p^{j-\mu}} e\left(\frac{d}{p^j}\left(n'-\frac{h^2-1}{24}\right)\right) \\ 
	\times \sum_{\ell (p^{j-\mu})} e\pfrac{-d(6p^\mu+\overline{(4m'/p^\mu)})\ell^2-dh\ell}{p^{j-\mu}}.
\end{multline}
Since 
\[
	\overline{(4m'/p^\mu)}m = \overline{(4m'/p^\mu)}(24p^\mu m'/p^\mu+1) \equiv 6p^\mu + \overline{(4m'/p^\mu)} \pmod{p^{j-\mu}},
\] 
the inner sum is equal to
\[
	G(-dm\overline{(4m'/p^\mu)},-dh,p^{j-\mu}).
\]

We first consider the case where $p\nmid m$, and we choose $\bar m$ such that $\bar m m\equiv 1\pmod{24p^\lambda}$.
Using \eqref{eq:gauss-4c-eval} again we find that
\[
	F_h(p^j) = p^j \pfrac{-m}{p^{j-\mu}}\ep_{p^{j-\mu}}^2 \sum_{d(p^j)^*} e\left(\frac{d}{p^j}\left(n'-\frac{h^2-1}{24}+\bar mm'h^2\right)\right).
\]
Note that
\[
	n'-\frac{h^2-1}{24} + \bar mm'h^2 
	\equiv \frac{n-\bar mh^2}{24} \pmod{p^\lambda}.	
\]
Since $\ep_{p^{j-\mu}}^{2}=\pfrac{-1}{p^{j-\mu}}$ and $m\equiv 1\pmod{p^\mu}$
we conclude that the quantity in \eqref{eq:F-h-c-prime-power} is
\begin{align*}
	p^{-\lambda} \sum_{j=0}^\lambda \pfrac{m}{p^{\lambda-j}} p^{-j} F_h(p^j) &= p^{-\lambda} \pfrac{m}{p^\lambda} \sum_{j=0}^\lambda \sum_{d(p^j)^*} e\left(\frac{dp^{\lambda-j}}{p^\lambda}\left(\frac{n-\bar m h^2}{24}\right)\right) \\
	&= p^{-\lambda} \pfrac{m}{p^\lambda} \sum_{d(p^\lambda)} e\left(\frac{d}{p^\lambda}\left(\frac{n-\bar m h^2}{24}\right)\right) \\
	&= 
	\begin{dcases}
		\pfrac{m}{p^\lambda} &\text{ if } p^\lambda \mid \frac{n-\bar m h^2}{24},\\
		0 & \text{ otherwise.}
	\end{dcases}
\end{align*}
The condition $p^\lambda \mid (n-\bar mh^2)/24$ is equivalent to $h^2\equiv mn\pmod{24p^\lambda}$, so \eqref{eq:F-h-c-prime-power} is true in the case where $p$ is an odd prime not dividing $m$.

We turn now to the case where $p\mid m$. Then $p\nmid m'$, so $\mu=0$ in \eqref{eq:F-h-p-j-two-sums}, and since $m$ is squarefree, $(m/p,p)=1$. Furthermore, all of the terms in the sum on the left-hand side of \eqref{eq:F-h-c-prime-power} vanish except for the term $j=\lambda$. From \eqref{eq:F-h-p-j-two-sums} we have
\[
	F_h(p^\lambda) = \ep_{p^\lambda} \pfrac{m'}{p^\lambda} p^{\lambda/2} \sum_{d(p^\lambda)^*} \pfrac{d}{p^\lambda} e\left(\frac{d}{p^\lambda}\left(n'-\frac{h^2-1}{24}\right)\right) G(-dm\bar{(4m')},-dh,p^\lambda),
\]
which is zero unless $p \mid h$. Assume that $p \mid h$; then, using \eqref{eq:gauss-4c-eval} we obtain
\begin{equation*}
	F_h(p^\lambda) = p^{\lambda+\frac12} \ep_p \pfrac{m'}{p} \pfrac{-m/p}{p^{\lambda-1}} \sum_{d(p^\lambda)^*} \pfrac{d}{p} e\left(\frac{d}{p^{\lambda}}\left(n'-\frac{h^2-1}{24}+\frac{\bar{(m/p)}m'h^2}{p} \right)\right).
\end{equation*}
Set 
\[
	\alpha:= n'-\frac{h^2-1}{24}+\frac{\bar{(m/p)}m'h^2}p.
\]
Replacing $d$ by $d+p$, we see that the sum
\[
	\sum_{d(p^\lambda)^*} \pfrac{d}{p} e\pfrac{d\alpha}{p^\lambda}
\]
is zero unless $p^{\lambda-1} \mid \alpha$. Assume that $p^{\lambda-1}\mid\alpha$. Then
\begin{align*}
	\sum_{d(p^\lambda)^*} \pfrac{d}{p} e\pfrac{d\alpha}{p^\lambda} 
	= p^{\lambda-1} \sum_{d(p)} \pfrac{d}{p} e\pfrac{d\alpha/p^{\lambda-1}}{p}
	= \ep_p \, p^{\lambda-\frac12} \pfrac{\alpha/p^{\lambda-1}}{p},
\end{align*}
where the last equality uses Theorem 1.1.5 of \cite{BEW:Gauss} and \eqref{eq:gauss-eval}.
We have $p\neq 3$ since $m\equiv 1 \pmod{24}$, so $\pfrac{m'}{p}=\pfrac{-24}{p}$.
Therefore
\[
	F_h(p^\lambda) = p^{2\lambda} \ep_p^2 \pfrac{-m/p}{p^{\lambda-1}} \pfrac{-24\alpha/p^{\lambda-1}}{p}.
\]
We have
\begin{align*}
	\frac{24\alpha}{p^{\lambda-1}} = \frac{p(n-h^2)+\bar{(m/p)}24m'h^2}{p^{\lambda}}
\end{align*}
which, together with the fact that $\overline{(m/p)}m\equiv p\pmod{p^{\lambda+1}}$, yields
\[
	\frac{24\alpha}{p^{\lambda-1}} \equiv \frac{\overline{(m/p)}\left[mn-mh^2+24m'h^2\right]}{p^\lambda} \equiv \frac{\overline{(m/p)}\left[mn-h^2\right]}{p^\lambda} \pmod{p}.
\]
Therefore
\begin{align*}
	F_h(p^\lambda) = p^{2\lambda} \pfrac{-m/p}{p^{\lambda}} \pfrac{(h^2-mn)/p^\lambda}{p}
	= p^{2\lambda} \pfrac{m/p^*}{p^{\lambda}} \pfrac{p^*}{(h^2-mn)/p^\lambda},
\end{align*}
under the assumption that $p^{\lambda-1}\mid \alpha$. This assumption is equivalent to $h^2\equiv mn\pmod{24p^\lambda}$ and implies that $p\mid h$, which justifies our previous assumption. Thus we conclude that
\begin{equation}
	F_h(p^\lambda) = 
	\begin{dcases}
		p^{2\lambda} \pfrac{m/p^*}{p^\lambda}\pfrac{p^*}{(h^2-mn)/p^\lambda} &\text{ if }h^2\equiv mn\pmod{24p^\lambda},\\
		0 & \text{ otherwise},
	\end{dcases}
\end{equation}
which verifies \eqref{eq:F-h-c-prime-power} in the case where $p$ is an odd prime dividing $m$.

Now suppose that $p=2$.
Since $2\nmid m$ and $\pfrac m2=1$, we want to show that
\begin{equation} \label{eq:p=2-want}
	\sum_{j=0}^\lambda 2^{-j} F_h(2^j) = 
	\begin{cases}
		2^\lambda &\text{ if } h^2\equiv mn\pmod{24\cdot2^\lambda},\\
		0 & \text{ otherwise}.
	\end{cases}
\end{equation}
We recall the definition
\[
	F_h(u) := \sum_{d(u)^*}e\left(\frac du\left(n'-\frac{h^2-1}{24}\right)\right)\sum_{k(u)} e\left(-\frac{dk(6k+h)}{u}\right) G(\bar dm', k, u).
\]
Define $\mu$ by
\[
	2^\mu=(m',2^j).
\]
Then
\[
	G(\bar dm',k,2^j) = 
	\begin{cases}
		2^\mu G(\bar dm'/2^\mu,k/2^\mu,2^{j-\mu}) &\text{ if }2^\mu \mid k,\\
		0 & \text{ otherwise}.
	\end{cases}
\]
Let 
\[
	\beta:= n'-\frac{h^2-1}{24}+h^2\,\bar m\,m', \text{ with } \bar m\,m\equiv 1\pmod{24\cdot2^\lambda}.
\]
We claim that
\begin{equation} \label{eq:F-h-2-claim}
	F_h(2^j) = 2^j \sum_{d(2^j)^*} e\pfrac{d\beta}{2^j}.
\end{equation}
If $\mu=j$ then $G(\bar dm'/2^\mu,k/2^\mu,2^{j-\mu})=1$ and $2^j\mid m'$, so
\[
	F_h(2^j) = 2^j \sum_{d(2^j)^*} e\pfrac{d\beta}{2^j}.
\]
If $\mu=j-1$ then
\[
	G(\bar dm'/2^\mu,k/2^\mu,2^{j-\mu}) = 
	\begin{cases}
		2 &\text{ if } k/2^\mu \text{ is odd},\\
		0 &\text{ if } k/2^\mu \text{ is even}.
	\end{cases}
\]
Since $2^{j-1}\mid m'$ and $\bar m$ is odd, we have $\beta\equiv n'-(h^2-1)/24-2^{j-1}h \pmod{2^j}$, which yields
\begin{align*}
	F_h(2^j) &= 2^j \sum_{d(2^j)^*} e\left(\frac{d}{2^j}\left(n'-\frac{h^2-1}{24}\right) - \frac{d\cdot 2^{j-1}(6\cdot 2^{j-1}+h)}{2^j}\right) \\
	&=2^j \sum_{d(2^j)^*} e\pfrac{d\beta}{2^j}.
\end{align*}
If $\mu \leq j-2$ then by \eqref{eq:gauss-4c-eval} we have
\begin{multline*}
	G(\bar dm'/2^\mu,k/2^\mu,2^{j-\mu}) = \\ 
	\begin{dcases}
		(1+i)\,\ep_{\bar dm'/2^\mu}^{-1} \, 2^{\frac{j-\mu}2} \pfrac{2^{j-\mu}}{\bar dm'/2^\mu} e\left( -\frac{d(\bar{m'/2^\mu})(k/2^\mu)^2}{2^{j-\mu+2}} \right) &\text{ if $k/2^\mu$ is even},\\
		0 &\text{ if $k/2^\mu$ is odd}.
	\end{dcases}
\end{multline*}
Writing $k=2^{\mu+1}\ell$, we have
\begin{multline*}
	F_h(2^j) = (1+i)\,2^{\frac{j+\mu}{2}} \sum_{d(2^j)^*} e\left(\frac{d}{2^j}\left(n'-\frac{h^2-1}{24}\right)\right) \ep_{\bar dm'/2^\mu}^{-1} \pfrac{2^{j-\mu}}{\bar dm'/2^\mu} \\*
	\times \sum_{\ell(2^{j-\mu-1})}e\left(-\frac{d\cdot 2^{\mu+1}\ell(6\cdot 2^{\mu+1}\ell+h)}{2^j} - \frac{d(\bar{m'/2^\mu})\ell^2}{2^{j-\mu}}\right).
\end{multline*}
Since $(\bar{m'/2^\mu})+24\cdot 2^\mu\equiv(\bar{m'/2^\mu})m\pmod{2^{j-\mu}}$, the inner sum equals
\begin{multline*}
	\frac 12 G(-dm(\bar{m'/2^\mu}),-2dh,2^{j-\mu}) \\
	= (1+i) \, 2^{\frac{j-\mu}{2}-1} \, \ep_{-dm(\bar{m'/2^\mu})}^{-1} \pfrac{2^{j-\mu}}{-dm(\bar{m'/2^\mu})} e\pfrac{dh^2\,\bar m\,m'}{2^j}.
\end{multline*}
Since
\[
	\ep_{\bar dm'/2^\mu}^{-1} \ep_{-dm\overline{(m'/2^\mu)}}^{-1} = -i,
\]
we have
\[
	F_h(2^j) = 2^j \sum_{d(2^j)^*} e\pfrac{d\beta}{2^j}.
\]

We conclude in every case that
\begin{align*}
	\sum_{j=0}^{\lambda} 2^{-j} F_h(2^j) &= \sum_{j=0}^\lambda \sum_{d\,(2^j)^*} e\pfrac{d\beta}{2^j} =
	\sum_{d(2^\lambda)}e\pfrac{d\beta}{2^\lambda}=
	\begin{cases}
		2^\lambda &\text{ if }h^2\equiv mn\pmod{24\cdot 2^\lambda},\\
		0 & \text{ otherwise},
	\end{cases}
\end{align*}
which completes the proof.
\end{proof}

\section{Proof of Theorem \ref{thm:main}} \label{sec:proof-of-theorem}

We begin by recording exact formulas for the coefficients $p(m,n)$ in terms of Kloosterman sums and the $I$- and $J$-Bessel functions $I_\alpha(x)$ and $J_\alpha(x)$. 
The following formulas are found in Proposition 11 of \cite{Ahlgren:2013aa}.
Let $\tilde{h}_m$ denote the functions in that paper; then our functions $h_m$ described in \eqref{eq:h-m-pos} and \eqref{eq:h-m-neg} are normalized as
\begin{equation} \label{eq:h-m-normalize}
	h_m = 
	\begin{cases}
		-\frac{3m^{3/2}}{4\sqrt{\pi}} \, \tilde h_m & \text{ if }m>0, \\
		\hphantom{-} |m|^{3/2} \, \tilde h_m & \text{ if } m<0.
 	\end{cases}
\end{equation}
Suppose that $m,n\equiv 1\pmod{24}$ are not both negative. By Propositions 8 and 11 of \cite{Ahlgren:2013aa} we have
\begin{equation} \label{eq:p-m-n-exact}
	p(m,n) =
	\begin{dcases}
		2\pi |mn|^{-1/4} \sum_{c>0} \frac{K(m',n';c)}{c} I_{3/2}\pfrac{\pi\sqrt{|mn|}}{6c} & \text{ if }mn<0, \\
		4(mn)^{-1/4} \sum_{c>0} \frac{K(m',n';c)}{c} \left[\frac{\partial}{\partial s} J_{s-1/2}\pfrac{\pi\sqrt{mn}}{6c} \bigg|_{s=2}\right] & \text{ if }mn>0.
	\end{dcases}
\end{equation}

To prove Theorem \ref{thm:main} we will show that the traces $\Tr_v(m,n)$ can also be expressed as infinite series involving Kloosterman sums and Bessel functions. 
This is essentially accomplished in the following proposition.
To simplify the statement of the proposition for square and non-square positive discriminants, we set $P_{v,Q}(\tau,s):=P_v(\tau,s)$ whenever $Q$ has positive non-square discriminant. The functions $P_v(\tau)$, $P_v(\tau,s)$, and $P_{v,Q}(\tau,s)$ are defined in \eqref{eq:P-v-tau-2=P-v-tau}, \eqref{eq:def-p-v-tau-s}, and \eqref{eq:def-p-v-q}, respectively.

\begin{proposition} \label{prop:traces-no-deriv}
Suppose that $m,n\equiv 1\pmod{24}$ and that $m$ is squarefree.
If $mn<0$ then
\begin{multline} \label{eq:mn-neg}
	|mn|^{-1/2} \sum_{Q\in \Gamma\backslash \sQ^{(1)}_{mn}} \chi_m(Q) P_v(\tau_Q)
	\\=
	2\pi |mn|^{-1/4} \sum_{d\mid v} \sqrt d \pfrac{12}{d} \pfrac{m}{v/d} \sum_{c>0} \frac{K((d^2m)',n';c)}c I_{3/2}\pfrac{\pi \sqrt{|d^2mn|}}{6c}.
\end{multline}
If $\re(s)>1$ and $m,n>0$ then
\begin{multline} \label{mn-pos-s}
	\frac{1}{2\pi} \sum_{Q\in \Gamma\backslash \sQ^{(1)}_{mn}} \chi_m(Q) \int_{C_Q} P_{v,Q}(\tau,s) \frac{d\tau}{Q(\tau,1)}
	\\*= 
	4 (mn)^{-1/4} \sum_{d \mid v} \sqrt d \pfrac{12}{d} \pfrac{m}{v/d} \sum_{c>0} \frac{K((d^2m)',n';c)}{c} J_{s-1/2} \pfrac{\pi \sqrt{d^2mn}}{6c}.
\end{multline}
\end{proposition}

Before proving Proposition \ref{prop:traces-no-deriv}, we remark that when $s=2$, the right-hand side of \eqref{mn-pos-s} is often identically zero. This follows from equation (3.15) of \cite{Ahlgren:2013aa}, which states that
\[
	\sum_{c>0} \frac{K(m',n';c)}{c} J_{3/2}\pfrac{\pi \sqrt{mn}}{6c} = 
	\begin{cases}
		0 & \text{ if } m\neq n, \\
		\frac 1{2\pi} & \text{ if } m=n.
	\end{cases}
\]
The only situation in which the right-hand side of \eqref{mn-pos-s} does not vanish is when $n=mt^2$ for some integer $t$ and $v=t\ell$, for some integer $\ell$ with $(\ell,m)=1$. In that case \eqref{mn-pos-s} becomes
\[
	\sum_{Q\in \sQ_{(mt)^2}^{(1)}} \chi_m(Q) \int_{C_Q} P_{v,Q}(\tau,2) \frac{d\tau}{Q(\tau,1)} = \frac{4}{\sqrt m} \pfrac{12}{t}\pfrac{m}{\ell}.
\]

\begin{proof}[Proof of Proposition \ref{prop:traces-no-deriv}]
Suppose that $mn<0$, and let $L_v^-(m,n)$ denote the left-hand side of \eqref{eq:mn-neg}. Using the definition of $P_v(\tau)=P_v(\tau,2)$ in \eqref{eq:def-p-v-tau-s}, we find that
\[
	L_v^-(m,n) = \tfrac1{6} |mn|^{-1/2} \sum_{d \mid 6}\sum_{Q\in \Gamma\backslash\sQ_{mn}^{(1)}} \sum_{\gamma\in \Gamma_\infty \backslash \Gamma} \mu(d) \, \chi_m(Q) \, \phi_v\left(\gamma W_d\,\tau_Q,2,0\right).
\]
Using P1 and P3 of Lemma \ref{lem:chi-props} and equation \eqref{eq:root-compatible}, this becomes
\[
	L_v^-(m,n) = \tfrac1{6} |mn|^{-1/2} \sum_{d \mid 6}\sum_{Q\in \Gamma\backslash\sQ_{mn}^{(1)}} \sum_{\gamma\in \Gamma_\infty \backslash \Gamma} \mu(d) \, \chi_m(\gamma W_d\, Q) \, \phi_v\left(\tau_{\gamma W_d\, Q},2,0\right).
\]
By \eqref{eq:W-d-r-r'} and \eqref{eq:Qn-decomp-W-d} the map $\left(\gamma,d,Q\right) \mapsto \gamma W_d\,Q$ gives a bijection 
\begin{equation} \label{eq:bij-neg-disc}
	\Gamma_\infty\backslash\Gamma \times \{1,2,3,6\} \times \Gamma\backslash \sQ^{(1)}_{mn} \longleftrightarrow \Gamma_\infty \backslash \sQ_{mn}.
\end{equation}
If $Q\in \sQ_{mn}^{(1)}$ and $Q'=W_dQ=[a,b,c]$ then $\mu(d)=\pfrac{12}b$ by \eqref{eq:W-d-r-r'} and \eqref{eq:r'-r-cases}.
With $\M_{s,k}(y)$ as in \eqref{def-M-s-k}, we have
\[
	L_v^-(m,n) = \tfrac1{6} |mn|^{-1/2} \sum_{\substack{Q\in \Gamma_\infty \backslash \sQ_{mn}\\Q=[a,b,c]}} \ptfrac{12}b \chi_m(Q) \M_{2,0}(4\pi v\im\tau_Q) e(-v\re\tau_Q).
\]
Since $\sQ_{mn}$ contains only positive definite forms (those with $a>0$) we have
\[
	\tau_Q = -\frac{b}{2a} + i \frac{\sqrt{|mn|}}{2a}.
\]
By (13.1.32) and (13.6.6) of \cite{AS:Pocketbook}, we have
\[
	\M_{2,0}(4\pi v y) = M_{0,3/2}(4\pi v y) = 12\pi\sqrt{v y} \, I_{3/2}(2\pi v y),
\]
which gives
\begin{equation} \label{eq:L-v-minus-Gamma-infty}
	L_v^-(m,n) = \pi \sqrt{2v} \, |mn|^{-1/4} \sum_{\substack{Q\in \Gamma_\infty \backslash \sQ_{mn}\\Q=[a,b,c]}} \pfrac{12}b \frac{\chi_m(Q)}{\sqrt a} I_{3/2}\pfrac{\pi v\sqrt{|mn|}}{a} e\pfrac{bv}{2a}.
\end{equation}

Now suppose that $m$ and $n$ are both positive, and let $L_v^+(m,n)$ denote the left-hand side of \eqref{mn-pos-s}.
As in \eqref{eq:d-tau-q-def}, we write
\[
	d\tau_Q = \frac{\sqrt{mn}\,d\tau}{Q(\tau,1)}.
\]
If $mn$ is not a square, then by \eqref{eq:def-p-v-tau-s} we have
\begin{equation} \label{eq:L-v-plus-C-Q}
	L_v^+(m,n) = \frac{C(s)}{2\pi \Gamma(2s) \sqrt{mn}} \sum_{d\mid 6} \mu(d) 
	\sum_{Q\in \Gamma \backslash \sQ^{(1)}_{mn}} \chi_m(Q) \sum_{\gamma\in \Gamma_\infty\backslash\Gamma} \int_{C_Q} \phi_v(\gamma W_d\, \tau,s,0) \, d\tau_Q.
\end{equation}
For each $Q$, let $\Gamma_Q \subseteq \Gamma$ denote the stabilizer of $Q$. We rewrite the sum over $\Gamma_\infty \backslash \Gamma$ as a sum over $\gamma\in \Gamma_\infty\backslash\Gamma/\Gamma_Q$ and a sum over $g\in \Gamma_Q$. Since $S_Q=\cup_{g\in \Gamma_Q} C_Q$, the inner sum in \eqref{eq:L-v-plus-C-Q} becomes
\[
	\sum_{\gamma\in \Gamma_\infty\backslash\Gamma/\Gamma_Q} \int_{S_Q} \phi_v(\gamma W_d\, \tau,s,0) \, d\tau_Q.
\]
In each integral, we replace $\tau$ by $W_d^{-1}\gamma^{-1}\tau$. Using \eqref{eq:d-tau-prime} and P1 and P3 of Lemma \ref{lem:chi-props}, we obtain
\begin{equation} \label{eq:L-v-plus-gamma-W-d}
	L_v^+(m,n) = \frac{C(s)}{2\pi\Gamma(2s) \sqrt{mn}} \sum_{d\mid 6} 
	\sum_{Q\in \Gamma \backslash \sQ^{(1)}_{mn}} \sum_{\gamma\in \Gamma_\infty\backslash\Gamma/\Gamma_Q}\mu(d) \chi_m(\gamma W_d\, Q) \int_{S_{\gamma W_d\,Q}} \phi_v(\tau,s,0) \, d\tau_{\gamma W_d\,Q}.
\end{equation}
As in \eqref{eq:bij-neg-disc}, we have the bijection
\[
	\Gamma_\infty\backslash \Gamma /\Gamma_Q \times \{1,2,3,6\} \times \Gamma\backslash \sQ^{(1)}_{mn} \longleftrightarrow \Gamma_\infty \backslash \sQ_{mn}
\]
given by $(\gamma,d,Q)\mapsto \gamma W_d\,Q$. Thus
\begin{equation} \label{eq:L-v-plus-Gamma-infty}
	L_v^+(m,n) = \frac{C(s)}{2\pi\Gamma(2s) \sqrt{mn}} \sum_{\substack{Q\in \Gamma_\infty \backslash \sQ_{mn} \\ Q=[a,b,c]}} \pfrac{12}{b} \chi_m(Q) \int_{S_Q} \M_{s,0}(4\pi v\im \tau) e(-v\re \tau) \, d\tau_Q.
\end{equation}
In order to treat the square case together with the non-square case, we will show that, with the added condition $a\neq 0$ in the sum, \eqref{eq:L-v-plus-Gamma-infty} holds when $mn$ is a square. 
Of course, $a\neq 0$ is implied in \eqref{eq:L-v-plus-Gamma-infty} when $mn$ is not a square.

Suppose that $mn>0$ is a square.
For each $Q\in \sQ_{mn}^{(1)}$ define $\mathfrak a_{i,Q}:=r_{i,Q}/s_{i,Q}$ as in \eqref{eq:def-p-v-q}.
The stabilizer $\Gamma_Q$ is trivial for all $Q\in \sQ_{mn}$ and, using \eqref{eq:def-p-v-q}, we have \eqref{eq:L-v-plus-gamma-W-d} with the added condition $\gamma W_d\, \mathfrak a_{i,Q}\neq \infty$ on the third sum; that is,
\begin{equation} \label{eq:L-v-plus-gamma-W-d-square}
	L_v^+(m,n) = \frac{C(s)}{2\pi\Gamma(2s) \sqrt{mn}} \sum_{d\mid 6} 
	\sum_{Q\in \Gamma \backslash \sQ^{(1)}_{mn}} \sum_{\substack{\gamma\in \Gamma_\infty\backslash\Gamma \\ \gamma W_d \, \mathfrak a_{i,Q} \neq \infty}}\mu(d) \chi_m(\gamma W_d\, Q) \int_{S_{\gamma W_d\,Q}} \phi_v(\tau,s,0) \, d\tau_{\gamma W_d\,Q}.
\end{equation}
The quadratic forms $Q$ having $\infty=[1,0]$ as a root are of the form $Q=[0,\pm b,*]$, where $b=\sqrt{mn}$. 
Thus the condition $\gamma W_d \, \mathfrak a_{i,Q} \neq \infty$ is equivalent to $\gamma W_d\, Q\neq [0,\pm b,*]$.
Applying the bijection \eqref{eq:bij-neg-disc}, which holds in this case since $\Gamma_Q$ is trivial, we obtain \eqref{eq:L-v-plus-Gamma-infty} with the restriction $a\neq0$ on the sum.

Treating the square and non-square case together, we assume that $m$ and $n$ are arbitrary positive integers satisfying $m,n\equiv 1\pmod{24}$. Suppose that $Q=[a,b,c]$. The apex of the semicircle $S_Q$ is 
\[
	-\frac{b}{2a} + i \frac{\sqrt{mn}}{2|a|},
\]
so we parametrize $S_Q$ by
\begin{align}
	\tau &= -\frac{b}{2a} + \frac{\sqrt{mn}}{2a} e^{i\sgn(a)\theta} \notag \\
	 &= -\frac{b}{2a} + \frac{\sqrt{mn}}{2a} \cos\theta + i \frac{\sqrt{mn}}{2|a|} \sin\theta, \qquad 0\leq \theta\leq \pi. \label{eq:parametrize-S_Q}
\end{align}
Then we have
\begin{align*}
	Q(\tau,1) &= a\left( \frac{b^2-2b\sqrt{mn}\,e^{i\sgn(a)\theta} + mn \, e^{2i\sgn(a)\theta}}{4a^2} \right) + b \left( \frac{-b+\sqrt{mn} \, e^{i\sgn(a)\theta}}{2a} \right) + c \\
		&= \frac{mn}{4a} (e^{2i\sgn(a)\theta}-1),
\end{align*}
which gives
\begin{equation} \label{eq:d-tau-Q-theta}
	d\tau_Q = \frac{\sqrt{mn}\,d\tau}{Q(\tau,1)} = \frac{d\theta}{\sin\theta}.
\end{equation}
Combining \eqref{eq:L-v-plus-Gamma-infty}, \eqref{eq:parametrize-S_Q}, and \eqref{eq:d-tau-Q-theta}, we obtain
\begin{equation*} \label{eq:l-v-plus-param}
	L_v^+(m,n) = \frac{C(s)}{2\pi\Gamma(2s) \sqrt{mn}} \sum_{\substack{Q\in \Gamma_\infty \backslash \sQ_{mn} \\ Q=[a,b,c],a\neq 0}} R_Q(m,n),
\end{equation*}
where
\[
	R_Q(m,n) := \pfrac{12}{b} \chi_m(Q) e\pfrac{bv}{2a}
	\int_0^\pi \M_{s,0}\left(\frac{2\pi v\sqrt{mn}}{|a|}\sin\theta\right) e\left(-\frac{v\sqrt{mn}}{2a}\cos\theta\right) \frac{d\theta}{\sin\theta}.
\]
For each $Q=[a,b,c]\in \Gamma_\infty \backslash \sQ_{mn}$ with $a>0$, we have (using P5 of Lemma \ref{lem:chi-props})
\begin{multline*}
	R_Q(m,n) + R_{-Q}(m,n) 
	\\= 2\pfrac{12}{b}\chi_m(Q)e\pfrac{bv}{2a} \int_0^\pi \M_{s,0}\left(\frac{2\pi v\sqrt{mn}}{a} \sin\theta \right) \cos\left(\frac{\pi v\sqrt{mn}}{a} \cos\theta \right) \frac{d\theta}{\sin\theta}.
\end{multline*}
By (13.6.6) of \cite{AS:Pocketbook} we have
\[
	\M_{s,0}(y) = M_{0,s-1/2}(y) = 2^{2s-1} \, \Gamma(s+1/2) \, \sqrt{y} \, I_{s-1/2}(y/2),
\]
hence
\begin{multline}
	L_v^+(m,n) = \frac{2^{2s-1/2} C(s) \Gamma(s+1/2)}{\sqrt\pi\,\Gamma(2s)(mn)^{1/4}} \sum_{Q\in \Gamma_\infty \backslash \sQ_{mn}^+} \pfrac{12}{b} \chi_m(Q) \sqrt{\frac va} \, e\pfrac{bv}{2a} 
	\\*
	\times \int_0^\pi I_{s-1/2}\left(\frac{\pi v\sqrt{mn}}{a}\sin\theta\right) \cos\left(\frac{\pi v\sqrt{mn}}{a} \cos\theta\right) \frac{d\theta}{\sqrt{\sin\theta}},
\end{multline}
where $\sQ_{mn}^+$ consists of those $Q=[a,b,c]$ with $a>0$. Lemma 9 of \cite{DIT:CycleIntegrals} asserts that for $\re s>0$ we have
\[
	\int_0^\pi \cos(t\cos\theta) I_{s-1/2}(t\sin\theta) \frac{d\theta}{\sqrt{\sin\theta}} = 2^{s-1}\frac{\Gamma(s/2)^2}{\Gamma(s)} J_{s-1/2}(t).
\]
Since
\[
	\frac{2^{3s-3/2}C(s)\Gamma(s+1/2)\Gamma(s/2)^2}{\sqrt\pi\,\Gamma(2s) \, \Gamma(s)} = 2\sqrt2,
\]
we obtain
\begin{equation} \label{eq:L-v-plus-Bessel-J}
	L_v^+(m,n) = 2\sqrt{2v} \, (mn)^{-1/4} \sum_{Q\in\Gamma_\infty \backslash \sQ_{mn}^+} \pfrac{12}{b} \frac{\chi_m(Q)}{\sqrt a} \, e\pfrac{bv}{2a} J_{s-1/2}\pfrac{\pi v\sqrt{mn}}{a}.
\end{equation}

Summarizing \eqref{eq:L-v-minus-Gamma-infty} and \eqref{eq:L-v-plus-Bessel-J}, we have
\begin{equation} \label{L-v-pm-gamma-infty}
	L_v^\pm(m,n)
	= \sqrt{2v} \, |mn|^{-1/4} \sum_{\substack{Q\in \Gamma_\infty \backslash \sQ_{mn}^+ \\ Q=[a,b,c]}} \pfrac{12}{b} \frac{\chi_m(Q)}{\sqrt a} e\pfrac{bv}{2a} \varphi^\pm\pfrac{\pi v\sqrt{|mn|}}{a},
\end{equation}
where
\begin{gather*}
	\varphi^-(x) = \pi I_{3/2}(x), \\
	\varphi^+(x) = 2 J_{s-1/2}(x).
\end{gather*}
Since $\pmatrix 1k01 [a,b,c]=[a,b-2ka,*]$, we have a bijection
\begin{equation*} \label{eq:sQ-infty-bijection}
	\Gamma_\infty\backslash \sQ_{mn}^+ \longleftrightarrow \left\{(a,b) : a>0, \  6\mid a, \  0\leq b<2a\right\},
\end{equation*}
which gives
\begin{multline} \label{eq:L-v-pm-a}
	L_v^\pm(m,n) = \\ \sqrt{2v} \, |mn|^{-1/4} \sum_{\substack{a>0 \\ 6\mid a}} a^{-1/2} \varphi^\pm\pfrac{\pi v\sqrt{|mn|}}{a}  \sum_{\substack{b\bmod 2a \\ \frac{b^2-mn}{4a}\in \Z}} \pfrac{12}b \chi_m\left([a,b,\tfrac{b^2-mn}{4a}]\right) e\pfrac{bv}{2a}.
\end{multline}
We write $a=6c$ and find that the inner sum in \eqref{eq:L-v-pm-a} equals $\frac 12 S_v(m,n;24c)$ (see \eqref{eq:def-S-v}), so
\[
	L_v^\pm(m,n) = \frac{\sqrt v}{2\sqrt 3} |mn|^{-1/4} \sum_{c>0} \frac{S_v(m,n;24c)}{\sqrt c} \varphi^\pm\pfrac{\pi v\sqrt{|mn|}}{6c}.
\]
Applying Proposition \ref{prop:s-k}, we obtain 
\begin{multline*}
	L_v^\pm(m,n) = 2\sqrt{v} \, |mn|^{-1/4} \sum_{c>0} c^{-1/2} \varphi^\pm\pfrac{\pi v\sqrt{|mn|}}{6c} \\\times \sum_{u \mid (v,c)} \pfrac{12}{v/u} \pfrac{m}{u}  \,\sqrt{\frac uc} K\left(\left(\tfrac{v^2}{u^2}m\right)',u';c/u\right).
\end{multline*}
We replace $c$ by $cu$ and switch the order of summation to obtain
\begin{equation*}
	L_v^\pm(m,n) = 2|mn|^{-1/4} \sum_{u\mid v} \pfrac{12}{v/u} \pfrac mu \sqrt{\frac vu} \, \sum_{c>0} \frac 1c K\left(\left(\tfrac{v^2}{u^2}m\right)',u';c\right) \varphi^\pm \pfrac{\pi v/u \sqrt{|mn|}}{6c}.
\end{equation*}
Finally, letting $d=v/u$ we conclude that
\begin{equation*}
	L_v^\pm(m,n) = 
	2|mn|^{-1/4} \sum_{d\mid v} \sqrt d \pfrac{12}{d} \pfrac{m}{v/d} \sum_{c>0} \frac{K((d^2m)',n';c)}c \varphi^\pm\pfrac{\pi \sqrt{|d^2mn|}}{6c},
\end{equation*}
from which Proposition \ref{prop:traces-no-deriv} follows.
\end{proof}

Theorem \ref{thm:main} now follows easily from Proposition \ref{prop:traces-no-deriv}.

\begin{proof}[Proof of Theorem \ref{thm:main}]
As above, we let $P_{v,Q}(\tau,s):=P_v(\tau,s)$ when $Q$ has positive non-square discriminant.
Then the definition of the traces in \eqref{eq:def-traces-1} becomes
\[
	\Tr_v(m,n)=
	\begin{dcases}
		|mn|^{-1/2} \sum_{Q\in \Gamma\backslash\sQ_{mn}^{(1)}} \chi_m(Q) P_v(\tau_Q) &\text{ if }mn<0,\\
		\frac1{2\pi}\sum_{Q\in \Gamma\backslash\sQ_{mn}^{(1)}} \chi_m(Q) \int_{C_Q} \left[\frac{\partial}{\partial s}P_{v,Q}(\tau,s) \bigg|_{s=2}\right] \frac{d\tau}{Q(\tau,1)} &\text{ if }mn>0.
	\end{dcases}
\]
When $mn<0$, \eqref{eq:main-thm} follows immediately from \eqref{eq:p-m-n-exact} and \eqref{eq:mn-neg}. When $m,n>0$, 
we take the derivative of each side with respect to $s$, then set $s=2$. Comparing the resulting equation with \eqref{eq:p-m-n-exact} gives \eqref{eq:main-thm}.
\end{proof}

\section*{Acknowledgments}

The author is grateful to Jan Bruinier for encouraging him to investigate the arithmetic nature of the coefficients $p(m,n)$.

\bibliographystyle{plain}
\bibliography{bibliography}

\end{document}